\documentclass[lettersize,journal]{IEEEtran}
\usepackage{changes} 
\usepackage{tikz}
\usetikzlibrary{fit}

\usetikzlibrary{arrows.meta, positioning, backgrounds, automata, decorations, shapes, fit}
\usepackage{pgfplots}
\usepackage[outline]{contour}
\contourlength{1.2pt}

\tikzset{
	>=stealth,
}
\usepgfplotslibrary{fillbetween}
\pgfplotsset{compat = newest}
\usepgfplotslibrary{polar}
\usepgflibrary{shapes.geometric}
\usetikzlibrary{calc}
\pgfplotsset{colormap/violet}

\usepackage{pdfrender}
\usepackage{comment}
\usepackage{multirow}
\usepackage{booktabs}

\usepackage{subcaption}

\usepackage{siunitx}
\usepackage{xcolor}
\usepackage{hyperref}   

\usepackage{mathtools}
\usetikzlibrary{shapes, arrows, snakes, calc}

\usepackage{amsmath}
\usepackage{graphicx}
\usepackage{latexsym}
\usepackage{color}
\usepackage{amssymb}
\usepackage{accents}
\usepackage{tabularx}
\usepackage{fancyhdr}
\usepackage{verbatim}
\usepackage{multirow}
\usepackage{framed}

\usepackage{url}
\usepackage{float}
\usepackage{enumitem}
\usepackage{mathrsfs}
\usepackage{amsfonts}
\usepackage{listings}
\usepackage{amsthm}
\usepackage{grffile}
\usepackage{sidecap}
\usepackage{pbox}
\usepackage[english,ruled,vlined,linesnumbered,algo2e]{algorithm2e}
\usepackage[noend]{algpseudocode}
\usepackage{scalefnt}
\usepackage{nomencl}
\usepackage{etoolbox}

\makenomenclature
\renewcommand\nomgroup[1]{%
	\item[\bfseries
	\ifstrequal{#1}{A}{Indices and index sets}{%
		\ifstrequal{#1}{B}{Parameters}{%
			\ifstrequal{#1}{C}{Decision variables}{%
	}}}%
	]\vspace{0.1in}}

\usepackage{diagbox}
\usepackage{cuted}
\usepackage{array}
\usepackage{textcomp}
\usepackage{stfloats}
\usepackage{bm}
\usepackage{algorithm}

\newtheorem{theorem}{Theorem}

\newcommand{\cG}{{\cal G}}
\newcommand{\cB}{{\cal B}}
\newcommand{\cL}{{\cal L}}
\newcommand{\cD}{{\cal D}}

\newcommand{\cT}{{\cal T}}

\newcommand{\cC}{{\cal C}}

\newcommand{\st}{\mbox{s.t.}}

\allowdisplaybreaks

\renewcommand{\underbar}{\underaccent{\bar}}

\newcommand{\tcb}{\textcolor{black}}

\newfloat{model}{thp}{lop}
\floatname{model}{Model}

\def\BibTeX{{\rm B\kern-.05em{\sc i\kern-.025em b}\kern-.08em
		T\kern-.1667em\lower.7ex\hbox{E}\kern-.125emX}}
\usepackage{balance}
\begin{document}
	\title{Multi-period Power System Risk Minimization under Wildfire Disruptions}
	\author{Hanbin Yang, Noah Rhodes, Haoxiang Yang, Line Roald, Lewis Ntaimo%
		\thanks{Hanbin Yang and Haoxiang Yang are supported by the National Natural Science Foundation of China, under the grant numbers 72201232 and 72231008. Noah Rhodes and Line Roald are supported by the U.S. National Science Foundation ASCENT program under award 2132904 and by the Department of Energy, Office of Science, Office of Advanced Scientific Computing Research, Applied Mathematics program under Contract Number DE-AC02-06CH11347.}%
	}
	
	
	\maketitle
	
	\begin{abstract}
		\tcb{Natural wildfire becomes increasingly frequent as climate change evolves, posing a growing threat to power systems, while grid failures simultaneously fuel the most destructive wildfires. Preemptive de-energization of grid equipment is effective in mitigating grid-induced wildfires but may cause significant power outages during natural wildfires. This paper proposes a novel two-stage stochastic program for planning preemptive de-energization and solves it via an enhanced Lagrangian cut decomposition algorithm. We model wildfire events as stochastic disruptions with random magnitude and timing. The stochastic program maximizes the electricity delivered while proactively de-energizing components over multiple time periods to reduce wildfire risks. We use a cellular automaton process to sample grid failure and wildfire scenarios driven by realistic risk and environmental factors.} We test our method on an augmented version of the RTS-GLMC test case in Southern California and compare it with \tcb{four} benchmark cases, \tcb{including deterministic, wait-and-see, and robust optimization formulations as well as a comparison with prior wildfire risk optimization}. Our method reduces wildfire damage costs and load-shedding losses, and our nominal plan is robust against uncertainty perturbation.
	\end{abstract}
	
	\begin{IEEEkeywords}
		optimal power flow, stochastic mixed-integer programming, wildfire risk, de-energization, Lagrangian cut, decomposition algorithm.
	\end{IEEEkeywords}
	
	\nomenclature[A,01]{$\cB$}{set of buses;}
	\nomenclature[A,02]{$\cG$}{set of generators;}
	\nomenclature[A,03]{$\cL$}{set of transmission lines;}
	\nomenclature[A,04]{$\cD$}{set of load demand;}
	\nomenclature[A,05]{$\cC$}{set of load components, \(\cC = \cB \cup \cG \cup \cL\);}
	\nomenclature[A,06]{$\cT$}{set of time periods, \(\cT = \{1,2,\dots,T\}\);}
	\nomenclature[A,07]{$\Omega$}{index set of wildfire scenarios;}
	\nomenclature[A,08]{$I_c^\omega$}{set of affected components by component \(c\)'s ignition under scenario \(\omega \in \Omega\);}
	
	\nomenclature[B,01]{$D_{dt}$}{load \(d \in \cD\) at time period \(t \in \cT\);}
	\nomenclature[B,02]{$w_{d}$}{the priority level of load \(d \in \cD\);}
	\nomenclature[B,03]{$r_{c}$}{repair cost for component \(c \in \cC \);}  
	\nomenclature[B,04]{$\underbar{P}^G_g, \bar{P}^G_g$}{maximum and minimum generation limits of \(g \in \cG \);}  
	\nomenclature[B,05]{$W_{ij}$}{the thermal power flow limit of the line \((i,j) \in \cL \);}  
	\nomenclature[B,06]{$b_{ij}$}{the susceptance of the line \( (i,j) \in \cL \);}  
	\nomenclature[B,07]{$\underbar{\theta}, \bar{\theta}$}{the big-M values for voltage angle difference;} 
	\nomenclature[B,08]{$\tau^\omega$}{disruption time for scenario \(\omega \in \Omega\);}
	\nomenclature[B,09]{$u_{c}^\omega$}{1 if component \(c \in \cC \) faults under scenario \(\omega \in \Omega\), 0 otherwise;}
	\nomenclature[B,10]{$v_{c}^\omega$}{1 if component \(c \in \cC \) is shut off by exogenous fire under scenario \(\omega \in \Omega\), 0 otherwise;}

	\nomenclature[C,01]{$\theta_{it}$}{phase angle of the bus \(i \in \cB\) at time \(t \in \cT \);}
	\nomenclature[C,02]{$P^L_{ijt}$}{active power flow on the line \((i,j) \in \cL\) at time \(t \in \cT \);}
	\nomenclature[C,03]{$p^G_{gt}$}{active power generation at generator \(g \in \cG\) at time \(t \in \cT \);}
	\nomenclature[C,04]{$x_{dt}$}{percentage of demand satisfied at the load \(d \in \cD\) at time \(t \in \cT \);}
	\nomenclature[C,05]{$z_{ct}$}{$1$ if component $c\in \cC$ is functional at time $t \in \cT$, $0$ otherwise;}
	\nomenclature[C,06]{$y^\omega_{ct}$}{$1$ if component $c \in \cC$ is functional at time $t\in \cT$ under scenarios $\omega \in \Omega$, $0$ otherwise;}
	\nomenclature[C,07]{{$\eta^\omega_{ct}$}}{$1$ if a fire damage is incurred at component $c \in \cC$ at time $t \in \cT$ under scenarios $\omega  \in \Omega$, $0$ otherwise.}
	\printnomenclature
	
	\section{Introduction}
	{Wildfire-related disasters have become more frequent and severe in recent years across the United States~\cite{Muhs2020, Holmes2008}. A major aspect of the wildfire research focuses on wildfire mitigation plans in power systems~\cite{vazquez2022wildfire}. There is a mutual impact between wildfire ignition and power system operations: on the one hand, many catastrophic fires with significant loss of life and property were sparked by power infrastructures; for example, the Camp Fire in 2018, sparked by a transmission line, caused $84$ lives deaths and an estimated $\$9.3$ billion in residential property damage alone~\cite{2019wildfirerisk}; on the other hand, regardless of the source, wildfires lead to damages in power system components, which change the topology and resource availability within a power system~\cite{choobineh2015vulnerability}.}
	
	{Utility companies have considered approaches in different time scales to reduce the risk of wildfire ignitions~\cite{vazquez2022wildfire}. Long-term approaches can last for years, such as equipment inspection, hardening, and vegetation management~\cite{WildfireManagement}, may require enormous costs and qualified personnel and should be complemented with day-to-day operations such as public safety power shut-offs (PSPS)~\cite{PSPS} to eliminate risk. In this paper, we focus on optimizing the operational de-energization decisions over a short time horizon similar to the setup in~\cite{rhodes2020balancing}.}
	
	When optimizing the de-energization decisions, many previous works mainly focus on wildfire risks caused by power system components but overlook the simultaneous natural/man-made wildfires' damage on power systems~\cite{taylor2022framework,muhs2020wildfire, davoudi2021reclosing}. In addition, previous literature and industry practices take a deterministic approach because it is computationally challenging to solve optimization under uncertainty models with binary variables for de-energization, such as a wildfire hazard prevention system based on data mining techniques \cite{DataMiningPreventionWildfire}, risk-based optimization models for de-energization~\cite{rhodes2020balancing, bayani2022quantifying}, a data-driven wildfire decision-making framework for power system resilience~\cite{Hong2022data}, and industry practices using thresholds to indicate when to shut off~\cite{vazquez2022wildfire, huang2023review}. However, incorporating wildfire spread dynamics can significantly improve the performance of de-energization. Recently, more models have started to incorporate stochastic aspects of wildfire risks, such as a dynamic programming model to optimize a PSPS~\cite{DP_PSPS}, a two-stage robust optimization model for capacity expansion considering PSPS~\cite{bayani2023resilient}, a two-stage stochastic integer program solved by progressive hedging~\cite{zhou2023optimal}, and {rolling horizon optimization models for joint de-energization and restoration~\cite{Noah2022, kody2022sharing}}. That said, most of the literature still carries a relatively simple decision structure and cannot fully capture the temporal dynamics of the wildfire; for example, a commonly used uncertainty set in robust optimization, such as a budgeted uncertainty set~\cite{ke2022managing}, can only capture the number of damaged components, but is not able to accurately characterize the temporal effect of wildfire progress within a multi-period horizon. Therefore, it requires new studies to examine the interaction between power system operations with high-fidelity and dynamic wildfire progress~\cite{WildfireManagement2}.
	
	{To holistically analyze the two-way impact between the wildfire ignitions and the power system operations, we} classify wildfires into two categories, \textit{exogenous} and \textit{endogenous} fires: exogenous fires are external to the power system, while power system faults trigger endogenous fires. {Extreme weather conditions such as drought~\cite{Li2021SpatialAT} can lead to exogenous wildfires, which may cause damages in power system components; e.g., dysfunctional shunts~\cite{nagpal2018wildfire}, utility-scale converters~\cite{pico2019transient}, and line conductors~\cite{mitchell2013power}.} Endogenous fires are ignited when downed transmission lines come into contact with surrounding vegetation and can damage nearby power system components~\cite{rhodes2020balancing}. {A framework to assess endogenous wildfire risk is proposed in~\cite{Scott2013AWR}, and simulation models based on such risk measures can provide early warnings~\cite{8620983}. In this paper, we explicitly model both exogenous and endogenous wildfires' impact via binary parameters indicating whether either wildfire damages each power system component in every period.}
	
	Optimizing a sequential decision process that models shut-off and power flow decisions requires solving a multistage stochastic program. This can be challenging because the computational complexity increases exponentially with the planning horizon length~\cite{BirgLouv97}, and there may be integer variables involved~\cite{rhodes2020balancing}. We develop a two-stage stochastic mixed-integer programming (SMIP) model that captures the interaction between power system de-energization and both exogenous and endogenous wildfires over multiple periods. {We minimize our wildfire risk measure, the expected total cost that incorporates the wildfire damage cost and the load-shedding cost.} In addition, our two-stage SMIP model incorporates the spatial and temporal uncertainty of wildfires and represents real-world situations with disruption scenarios, which can handle complex and dynamic situations better and produce more resilient operations than deterministic models. We use a cellular automaton process to simulate wildfire spread and generate scenarios that describe the magnitude and timing of wildfire disruptions. Since wildfire occurs infrequently, we assume at most one wildfire occurrence during the planning horizon to reduce the multi-period problem into two stages~\cite{HaoxiangIJOC}. {This two-stage scenario-based setting fits the low-probability high-impact property of wildfire ignition and can capture the dynamics of wildfire progress with a high fidelity. We design a proactive scheme that executes an effective de-energization plan before the wildfire and optimizes power flow operations of the remaining available components afterward.}
	
	We propose a decomposition algorithm for {the two-stage} SMIP with non-convex and non-smooth value functions. Our algorithm uses linear cutting planes, named \emph{Lagrangian cuts}, to approximate the value functions~\cite{zou2016nested}. {The vanilla version of Lagrangian cuts is tight theoretically but can be inefficient computationally}. To overcome this limitation, we introduce an enhanced version of the Lagrangian cut, named \textit{square-minimization cut}, which improves the quality of the approximations and speeds up the algorithm. Our method shows high computational accuracy and reliability. The two-stage SMIP model supports the decision-making process under uncertainty, reduces the likelihood of a disaster, and maintains as much load delivery as possible after a disruption.

	The contributions of our paper are fourfold:
	{\begin{enumerate}
			\item We propose a cellular automaton model that simulates the onset and spread of exogenous and endogenous wildfires based on environmental data. This enables us to assess wildfire risk through scenarios used in the SMIP model.
			\item We formulate a two-stage SMIP model that manages wildfire risk in power system operations. Unlike the existing literature, our model captures the random nature of wildfire ignitions using disruption scenarios with random onset times to keep the model size manageable.
			\item We develop a cutting plane algorithm that solves the SMIP model with a finite convergence guarantee. We improve its computational performance significantly with a new square-minimization cut.
			\item We show our two-stage SMIP model yields significantly more resilient solutions \tcb{than its deterministic counterparts and more economical solutions than the robust optimization model.} We also analyze the solution under different parametric settings to gain operational insights.
	\end{enumerate}}
	The rest of the paper is organized as follows: Section~\ref{sec:model} presents the wildfire simulation model and the SMIP formulation; Section~\ref{sec:algorithm} derives the cutting plane algorithm and its computational enhancement; Section~\ref{sec:results} covers benchmark cases and numerical analyses; and Section~\ref{sec:conclusion} concludes the paper with a brief discussion of future work.

	\section{Two-Stage SMIP Model} \label{sec:model}
	We consider multi-period dispatch and de-energization operations under wildfire disruption over a short-term time horizon ($24$ hours). The power network is represented by a graph $(\cB, \cL)$, where $\cB$ is the set of buses and $\cL$ is the set of lines. We use $\mathcal{D}_i$, $\mathcal{G}_i$, and $\cL_i$ to represent the subset of loads, generators, and transmission lines connected to bus $i$, respectively. For each period $t \in \mathcal{T}$, we formulate DC power flow constraints to decide the active power generation \(P^G_{gt}\) of generator \(g \in \cG\), the power flow $P^L_{ijt}$ on transmission line \((i,j) \in \cL\) from bus \(i\) to bus \(j\), and the phase angle $\theta_{it}$ of bus \(i \in \cB\).

	In our model, we define wildfire risk as the combined cost of load-shedding and wildfire damage under uncertain wildfire disruption over the specified time horizon. In model~\eqref{prob:stage1}, a de-energized generator will have zero generation capacity, and a de-energized line will be considered open. If a bus is shut off, all generators and lines connected to it will also be de-energized. Our model uses binary decision variables $z_{ct}$ to represent whether a component $c \in \cC$ is de-energized at time period $t$. We assume that a de-energized component will remain off until the end of the time horizon due to safety considerations \cite{rhodes2020balancing}. Model~\eqref{prob:stage1} obtains a first-stage plan, i.e., a nominal plan, that should be implemented until a disruption occurs or the time horizon ends, whichever occurs first. If a fire disruption is observed at period $\tau^\omega$, given the current shut-off state $z_{\cdot \tau^\omega-1}$, the model enters the second stage and incurs the second-stage value function $f^\omega$, where we assume that the components at revealed ignition locations will be shut off.     
	
	\subsection{Wildfire Scenario Modeling} \label{subsec:scenario}
	On top of the constructed electric power network, we model wildfire uncertainty as a stochastic disruption with random occurrence timing, location, and spread. We consider a short-term operation and thus assume at most one disruption will occur over the time horizon. We form a portfolio of disruption scenarios indexed by $\omega \in \Omega$, within which we model i) endogenous fires, which are caused by faults of operating electric components, and ii) exogenous fires, which are caused by random natural wildfire ignitions. 
	
	We consider the smallest rectangular area that can cover the entire power network and divide it into a two-dimensional grid of a specified resolution, in which the set of cells is denoted by $\mathcal{K}$. We then construct a mapping between the grid cells and the buses and transmission lines based on their geographic locations. We let $k_i \in \mathcal{K}$ denote the cell where the bus $i$ is located. A transmission line, $l = (i,j)$, can span multiple cells, which is denoted by $\mathcal{K}_{ij}$. Furthermore, a cell $k \in \mathcal{K}$ can contain multiple transmission lines, and we denote the set of transmission lines passing cell $k$ by $\cL^k$. Finally, we let $\mathcal{K}_{\cB}$ and $\mathcal{K}_{\cL}$ be the collection of cells containing at least a bus and at least a transmission line, respectively.
	
	We employ the cellular automaton method, with propagation rules listed in~\cite{CellularAutomata}, to simulate the spread of {exogenous} wildfires. Overall, this setup leads to a wildfire disruption in approximately $95\%$ of scenarios. Each cell is characterized by four states which evolve in discrete time. The four states are as follows: (i) \textit{cell has no forest fuel; }(ii) \textit{cell has forest fuel that has not ignited;} (iii) \textit{cell has ignited forest fuel;} (iv) \textit{cell has fully burning forest fuel.} Endogenous disruptions are caused by power system component failures. The status of each component is normal, at fault, or ignited. We define the onset time of a wildfire disruption \(\tau^\omega\) as the earlier time between an exogenous wildfire occurrence and any electric component's fault. 

	\paragraph{Exogenous Wildfires} We assume that there are no exogenous wildfires in the grid at the start of the time horizon. For each cell $k \in \mathcal{K}$ in period $t \in \cT$, it receives information from the environment to update its state and interacts with surrounding cells {based on a set of defined rules: 
		\begin{enumerate}
			\item if its state in period $t-1$ is the unburned fuel state (ii), it has a probability $p_{tk}$, influenced by environmental factors, to be ignited and its state will transfer from (ii) to the ignited fuel state (iii);
			\item if its state in period $t-1$ is (iii), then the state will transfer to the fully burning state (iv);
			\item if it is a fully burning cell, the state will stay at (iv) and can spread flames to its $8$ adjacent cells indexed by $k'$, igniting them with probabilities $q_{tk'}$ if they are in status (ii), where parameters $q_{tk'}$ are subject to environmental factors like vegetation, ground elevation, and wind.
	\end{enumerate}}
	We record the exogenous wildfire damage after the end of time horizon $\cT$, letting $v_c^\omega = 1$ if a wildfire (state (iii) or (iv)) ever occurs on \(k_i\) if \(c\) is bus \(i\) or any generator connected to bus \(i\), or on any cell in \(\mathcal{K}_{ij}\), if \(c\) is transmission line \((i,j)\).    
	
	\paragraph{Endogenous Wildfires} We simulate endogenous wildfires using a two-step procedure:
	{\begin{enumerate}
			\item First, we simulate electrical component faults. To initialize the fault simulation process, we assume there is no fault component in the power system at the start of the time horizon, similar to the exogenous wildfire simulation. For every period \(t \in \cT\), a fault may occur at a component $c\in\cC$ following a Bernoulli distribution with probability $p_{tc}$. Once a component faults, we record the state by letting $u_c^\omega = 1$. 
			\item Our second step runs an independent cellular automaton process for each fault component \(c\) from the fault period to the end of the time horizon, initialized with the cell(s) that contains the component in state (iii). After the cell evolves to state (iv) in the next period, it will spread the simulated wildfire to the adjacent cell \(k'\) with a probability \(q_{t,k'}\) at time period \(t\) and further propagate. The components damaged by this spread in the time horizon form a set $I_c^\omega$. 
	\end{enumerate}}
	
	Fig.~\ref{fig: random set} {provides an illustration of a set of components affected by an endogenous fire: a fault occurs at an energized bus $i$, igniting an endogenous wildfire and impacting two transmission lines, a bus, and a generator (marked by dashed lines).} Notice that we assume a fault directly leads to a wildfire because we aim to simulate how an endogenous fire would spread if the component \(c\) were not de-energized. In addition, we build independent simulation processes, each starting with only the cell(s) containing the fault component \(c\), in order to clearly trace the spread of a wildfire to its origin and eliminate the interference of endogenous wildfires started elsewhere.
	\begin{figure}
		\centering
		\begin{tikzpicture}[scale = .7, square/.style={regular polygon,regular polygon sides=4},node distance={20mm}, thick, main/.style = {draw, circle}, mycircle/.style={
			circle,
			draw=black,
			fill=gray,
			fill opacity = 0.3,
			text opacity=1,
			inner sep=0pt,
			minimum size=20pt,
			font=\small},
		myarrow/.style={-Stealth},
		node distance=0.6cm and 1.2cm]
		

		\draw[ultra thick, red, ->] (8.5,6.5) arc (120:80:3) ;
		
		\node at (13,9) [mycircle,draw] (j1) {Bus $j_1$};
		\node at (9,5) [red, dashed, circle, draw] (i) {Bus $i$};
		\node at (12.5,4.5) [mycircle, dashed, blue, draw] (j3) {Bus $j_3$};
		\draw[dashed, blue, -] (i) -- (j1);
		\draw[-] (j1) -- (j3);
		\draw[dashed, blue, -] (i) -- (j3);
		
		\node at (7.5,4.5) [mycircle,draw] (g1) {$g_1$};
		\node at (10.5,5.5) [mycircle, dashed, blue, draw] (g2) {$g_2$};
		\draw[ -] (i) -- (g1);
		\draw[ blue, -] (i) -- (g2);
		
		\node at (13.8,7.8) [mycircle,draw] (g3) {$g_3$};
		\draw[ -] (j1) -- (g3);
		\node at (7.8,8.8) (T) {{$I_i^\omega = \{i,\  g_2,\ (i, j_1),\ (i, j_3),\ j_3\}$. }};
		\node at (9,7.3) (T) {{fire spread direction}};
		\end{tikzpicture}
		\caption{A random set $I_i^\omega$ illustration. {A fault occurs at bus $i$ (energized), which causes an endogenous wildfire to spread and affect two transmission lines} $(i,j_1)$ and $(i,j_2)$, bus $j_3$, and generator $g_2$, {which are marked by blue dashed lines}.}
		\label{fig: random set}
	\end{figure}
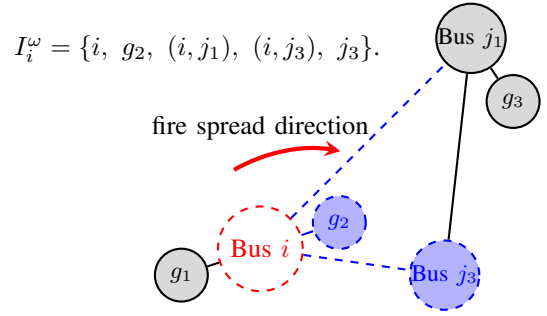 
	
	In summary, for a specific scenario $\omega \in \Omega$, we use a cellular automaton procedure to generate the wildfire disruption uncertainty \(\xi^\omega = \{\tau^\omega, v_{\cdot}^\omega, u_{\cdot}^\omega, I_{\cdot}^\omega\}\), which includes the disruption time \(\tau^\omega\), the binary parameters \(v_c^\omega\) indicating the exogenous fire locations, the binary parameters \(u_c^\omega\) indicating the fault locations, and the set \(I_c^\omega\) characterizing components damaged by the potential endogenous fire at component \(c\). We detail the parameter setup:
	\subsubsection*{Mapping onto Grid} 
	We convert the location of each bus $i\in\cB$ from latitude and longitude \cite{RTS-GLMC} to Universal Transverse Mercator(UTM) coordinates~\cite{maptools}. {The smallest rectangle that contains the entire power network has a length of $180,000$ and a width of $138,000$ UTM unit. We divide such a rectangular area into square cells with an edge length of \(1,000\) UTM units. Such edge length ensures a high resolution of our simulation model and is in the same order of magnitude as the average distance that a wildfire can spread within a time period.}
	
	\subsubsection*{Probability of Exogenous Wildfire Ignition $p_{tk}$}
	We assume that only wildfires that occur in cells with at least one power system component will threaten the electric power system. As stated in Section~\ref{subsec:scenario}, we model the ignition of exogenous wildfires at time \(t\) by a probability \(p_{tk}\), which describes the likelihood of a cell \(k \in \mathcal{K}\) transitioning from the state (ii), un-burned fuel, to state (iii), ignited the fuel.{To estimate \(p_{tk}\), we use the Wildland Fire Potential Index (WFPI)~\cite{WFPI} as a measure of the relative probability of ignition.}  The probability of an exogenous wildfire occurring in cell $k$ in which the set of all transmission lines passing through cell $k$ is $\cL^k$ is 
	\[p_{tk} = \dfrac{\sum_{l \in \cL^k}\texttt{WFPI}_l}{\sum_{l' \in \cL}{\texttt{WFPI}_{l'}}}.\]

	\subsubsection*{Probability of Wildfire Spread $q_{tk'}$} 
	Suppose there is a fully burning cell. Each of the adjacent cells $k' \in \mathcal{K}$ will be ignited with probability $q_{tk'}$. This probability \(q_{tk'}\) depends on: 1) a reference probability that cell $k'$ will be burning under no wind and flat terrain, $q_{0k'}$; 2) the vegetation type, $q^{\text {veg}}_{k'}$, and vegetation density, $q^{\text{den}}_{k'}$; 3) geographical information, $q^{s}_{k'}$; and (iv) wind speed and direction, $q^{w}_{tk'}$, as follows:
	\begin{equation*}
	q_{tk'}=q_{0k'}\left(1+q^{\text {veg}}_{k'}\right) \left(1 + q^{\text{den}}_{k'}\right)q^{s}_{k'} q^{w}_{tk'}.
	\end{equation*}
	{We refer to Ref.~\cite{CellularAutomata} and \cite{Alexandridis2008} for detailed parameter values.}

	\subsubsection*{Probability of Fault $p_{tc}$}     
	Transmission lines are more prone to failures than other electrical components due to environmental factors such as lightning and wind~\cite{PoissionRegression}. Therefore, in our implementation, we focus on the fault at transmission lines and set \(p_{tc} = 0, \forall c \in \cB \cup \cG\). This setup can be easily extended if we want to consider the risk of fault at buses and generators. In any time period, we assume a constant fault probability for each line $l \in \cL$ as \(p_{tl}=1-e^{-\bar{\lambda}}\), {where $\bar{\lambda}$ is the estimated hourly failure rate in the Poisson regression model of~\cite{PoissionRegression}.}

	\subsection{Two-Stage SMIP Model} \label{subsec:formulation}
	When a fault occurs at an energized component, the associated endogenous wildfire may incur high damage costs. To reduce the endogenous risk of igniting a wildfire, the power system operator can de-energize, i.e., shut off, power equipment. However, suppose multiple components are de-energized to prevent endogenous fires. In that case, the power supply capacity may be greatly reduced, resulting in the inability to meet crucial load demand and causing serious secondary disasters. To improve the reliability of the power system and mitigate wildfire damage, operators should de-energize some potentially dangerous electrical equipment under high-risk conditions to ensure enough power supply while significantly reducing risk. 
	
	To {achieve a balanced} risk trade-off between endogenous fire and load-shedding, we build a two-stage SMIP model~\eqref{prob:stage1} which optimizes shut-off decisions under wildfire disruption uncertainty over the given time horizon. In model~\eqref{prob:stage1}, a de-energized generator will have zero generation capacity, and a de-energized line will be considered open. If a bus is shut off, all generators and lines connected to it will also be de-energized. Our model uses binary decision variables $z_{ct}$ to represent whether a component $c \in \cC$ is de-energized at time period $t$. We assume that a de-energized component will remain off until the end of the time horizon due to safety considerations \cite{rhodes2020balancing}. Model~\eqref{prob:stage1} obtains a first-stage plan, i.e., a nominal plan, that should be implemented until a disruption occurs or the time horizon ends, whichever occurs first. If a fire disruption is observed at period $\tau^\omega$, given the current shut-off state $z_{\cdot \tau^\omega-1}$, the model enters the second stage and incurs the second-stage value function $f^\omega$, where we assume that the components at revealed ignition locations will be shut off.     
	\begin{subequations}
		\label{prob:stage1}
		\begin{align}
		& {Z^* = \min \ \sum_{\omega \in \Omega} p^\omega \left[ \sum_{t = 1}^{\tau^\omega-1} \sum_{d \in \cD} w_d (1-x_{dt}) + f^\omega(z_{\cdot \tau^\omega-1}, \xi^\omega) \right] }\notag \\
		& \st \quad \forall t \in \cT: \notag \\
		& P^L_{ijt} \leq -b_{ij} \left(\theta_{it} -\theta_{jt} + \bar{\theta} (1-z_{ijt})\right) \quad\quad\ \forall (i,j) \in \cL \label{eqn:pfconsl1}\\
		& P^L_{ijt} \geq -b_{ij} \left(\theta_{it} -\theta_{jt} + \underbar{\theta} (1-z_{ijt})\right) \quad \quad\ \forall (i,j) \in \cL \label{eqn:pfconsr1}\\
		& -W_{ij} z_{ijt} \leq P^L_{ijt} \leq W_{ij} z_{ijt} \quad \quad \quad \quad \quad \quad \forall (i,j) \in \cL \label{eqn:thermallimit1}\\
		& {\sum_{g \in \cG_i} P^G_{gt} + \sum_{l \in \cL_i} P^L_{lt} = \sum_{d \in \cD_i} D_{dt} x_{dt} \qquad\quad\quad\quad \forall i \in \cB} \label{eqn:flowbalance1}\\
		& \underbar{P}^G_g z_{gt} \leq P^G_{gt} \leq \bar{P}^G_g z_{gt} \quad\quad\quad\quad\quad\quad\quad\quad\quad\ \forall g \in \cG \label{eqn:genlimit1}\\
		& z_{it} \geq x_{dt} \quad\quad\quad\quad\quad\quad\quad\quad\quad\quad\quad\quad \; \forall i \in \cB, d \in \cD_i \label{eqn:loadlogic1}\\
		& z_{it} \geq z_{gt} \quad\quad\quad\quad\quad\quad\quad\quad\quad\quad\quad\quad\ \forall i \in \cB, g \in \cG_i \label{eqn:genlogic1}\\
		& z_{it} \geq z_{ijt} \quad\quad\quad\quad\quad\quad\quad\quad\quad\quad\quad \forall i \in \cB, (i,j) \in \cL \label{eqn:linelogic1}\\
		& z_{it} \geq z_{jit} \quad\quad\quad\quad\quad\quad\quad\quad\quad\quad\quad \forall i \in \cB, (j,i) \in \cL \label{eqn:linerevlogic1}\\
		& z_{ct} \geq z_{c \min\{t+1,T\}} \quad\quad\quad\quad\quad\quad\quad\quad\quad\quad\quad\ \forall c \in \cC \label{eqn:componenttime1} \\
		& z_{ct} \in \{0,1\}\quad\quad\quad\quad\quad\quad\quad\quad\quad\quad\quad\quad\quad\quad\ \forall c \in \cC.\label{eqn:binaryrestriction1}
		\end{align}
	\end{subequations}
	
	The objective function considers the expected cost, which consists of the load-shedding cost before the disruption and the total cost after the disruption. Constraints~\eqref{eqn:pfconsl1}-\eqref{eqn:genlimit1} correspond to the DC power flow model and constraints~\eqref{eqn:loadlogic1}-\eqref{eqn:linerevlogic1} model the component interactions, equivalent to constraint (7)-(9) in \cite{rhodes2020balancing}. Constraints~\eqref{eqn:componenttime1} describe the temporal logic of components' status: once a component is shut off, it stays off, and the corresponding \(z\) variable is $0$. 
	The scenario tree of this two-stage SMIP model is shown as Fig.~\ref{fig: scenario tree}, in which each node represents the decisions at a time period, and each branch extending to the right represents a disruption scenario. The {black} box of the left diagonal branch represents the nominal plan obtained if no disruption occurs. Suppose we have a scenario \(\omega\) with $\tau^\omega = 2$. The red box on the right describes the decisions associated with such value function \(f^\omega\){, and the red node implies the disruption time $\tau^\omega = 2$.} {The branches in the blue box are the scenarios that are disrupted in period $4$ and share the same historical information (the nodes in blue).}
	
	\begin{figure}
		\centering
		\begin{tikzpicture}[scale = .5, square/.style={regular polygon,regular polygon sides=4},node distance={20mm}, thick, main/.style = {draw, circle}, mycircle/.style={
			circle,
			draw=black,
			fill=gray,
			fill opacity = 0.3,
			text opacity=1,
			inner sep=0pt,
			minimum size=20pt,
			font=\small},
		myarrow/.style={-Stealth},
		node distance=0.6cm and 1.2cm]
		
		\node at (23.5,21) (Period) {Period};
		\node at (23.5,20) (Period1) {$1$};
		\node at (23.5,18.5) (Period2) {$2$};
		\node at (23.5,17) (Period3) {$3$};
		\node at (23.5,15.5) (Period4) {$4$};
		\node at (23.5,14) (Period5) {$5$};
		
		\node at (20.75, 16.25)  {$..$};
		
		\node at (17.75, 14.75)  {$..$};
		
		\node at (20,20) [circle, draw, blue] (R) {};
		
		\node at (21.5,18.5) [circle,draw, red] (11) {};
		\draw[-] (R) -- (11);
		\node at (18.5,18.5) [circle,draw, blue] (N1) {};
		\draw[-] (R) -- (N1);
		\node at (21.5,17) [circle,draw] (12) {};
		\draw[-] (11) -- (12);
		\node at (21.5,15.5) [circle,draw] (13) {};
		\draw[-] (12) -- (13);
		\node at (21.5,14) [circle,draw] (14) {};
		\draw[-] (13) -- (14);
		
		\node at (20,18.5) [circle, draw] (21) {};
		\draw[-] (R) -- (21);
		\node at (20,17) [circle, draw] (22) {};
		\draw[-] (21) -- (22);
		\node at (17,17) [circle, draw, blue] (N2) {};
		\draw[-] (N1) -- (N2);
		\node at (20,15.5) [circle, draw] (23) {};
		\draw[-] (23) -- (22);
		\node at (20,14) [circle, draw] (24) {};
		\draw[-] (23) -- (24);
		
		\node at (18.5,15.5) [circle,draw] (31) {};
		\draw[-] (N2) -- (31);
		\node at (15.5,15.5) [circle,draw] (N3) {};
		\draw[-] (N2) -- (N3);
		\node at (18.5,14) [circle,draw] (32) {};
		\draw[-] (31) -- (32);
		
		\node at (17,15.5) [circle,draw] (41) {};
		\draw[-] (N2) -- (41);
		\node at (17,14) [circle,draw] (42) {};
		\draw[-] (41) -- (42);
		\node at (14,14) [circle,draw] (N4) {};
		\node at (15.5,14) [circle,draw] (51) {};
		\draw[-] (N3) -- (51);
		\draw[-] (N3) -- (N4);
		
		\draw[-, dashed, black] (20.5,20) -- (14,13.5);
		\draw[-, dashed, black] (20,20.5) -- (13.5,14); 
		\draw[-, dashed, black] (20.5,20) -- (20,20.5);
		\draw[-, dashed, black] (13.5,14) -- (14,13.5);
		
		\draw[-, dotted, red] (22, 19) -- (22,13.5);
		\draw[-, dotted, red] (21, 19) -- (21,13.5);
		\draw[-, dotted, red] (21, 19) -- (22,19);
		\draw[-, dotted, red] (21, 13.5) -- (22,13.5);
		
		
		
		\draw[-, dotted, blue] (19, 16) -- (19,13.5);
		\draw[-, dotted, blue] (16.5, 16) -- (16.5,13.5);
		\draw[-, dotted, blue] (16.5, 16) -- (19,16);
		\draw[-, dotted, blue] (16.5, 13.5) -- (19,13.5);
		\node at (22.75,18.5){{$\overset{\tau^\omega}{\rightsquigarrow}$}};
		\end{tikzpicture}
		\caption{A scenario-tree illustration of the wildfire disruption problem with $T = 5$ and $\tau^\omega = 2$. {The branch in the red box represents the disruptive scenarios where the disruption occurs in period $2$. The branches in the blue boxes represent the scenarios sharing the same historical information (the nodes in blue).}}
		\label{fig: scenario tree}
	\end{figure}
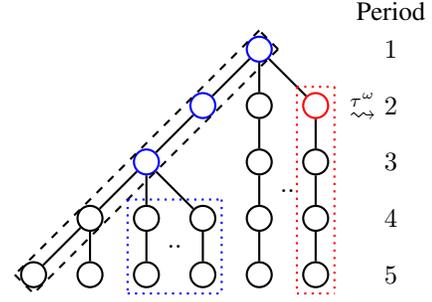   
	
	{We define the expected total cost after the disruption \(\sum_{\omega \in \Omega} p^\omega f^\omega(\cdot, \xi^\omega)\) as our wildfire risk measure. For each scenario, }{the second-stage value function \(f^\omega\) takes in the shut-off state variables $\hat{z}_{\cdot \tau^\omega-1}$ and the disruption uncertainty \(\xi^\omega\) as its input. It focuses on the operations after \(\tau^\omega\) and can be evaluated by solving the following mixed-integer program}:
	\begin{subequations}
		\label{prob:staget}
		\begin{align}
		& {f^\omega(\hat{z}_{\cdot \tau^{\omega}-1},\xi^\omega) = } \notag \\
		& \min \quad \sum_{t = \tau^\omega}^{T} \sum_{d \in \cD} w_d (1-x^\omega_{dt}) + \sum_{c \in \cC} r_c {\eta^\omega_{c}} \notag\\
		& \st \quad \forall t \in \{\tau^\omega,\dots, T\}: \notag \\
		& P^{L,\omega}_{ijt} \leq -b_{ij} \left(\theta^\omega_{it} -\theta^\omega_{jt} + \bar{\theta} (1-y^\omega_{ij}) \right) \qquad \forall (i,j) \in \cL\label{eqn:pfconslt}\\
		& P^{L,\omega}_{ijt} \geq -b_{ij} \left(\theta^\omega_{it} -\theta^\omega_{jt} + \underbar{\theta} (1-y^\omega_{ij})\right) \qquad \forall (i,j) \in \cL \label{eqn:pfconsrt}\\
		& -W_{ij} y^\omega_{ij} \leq P^{L,\omega}_{ijt} \leq W_{ij} y^\omega_{ij}\quad \quad \quad \quad \qquad \forall (i,j) \in \cL \label{eqn:thermallimitt}\\
		& \sum_{g \in \cG_i} P^{G,\omega}_{gt} + \sum_{l \in \cL_i} P^{L,\omega}_{lt} = \sum_{d \in \cD_i} D_{dt} x^\omega_{dt} \ \quad \qquad \forall i \in \cB \label{eqn:flowbalancet}\\
		& \underbar{P}^{G,\omega}_g y^\omega_{g} \leq P^{G,\omega}_{gt} \leq \bar{P}^{G,\omega}_g y^\omega_{g} \quad \quad \quad \quad \quad \quad \quad\ \ \forall g \in \cG \label{eqn:genlimitt}\\
		& y^\omega_i \geq x^\omega_{dt} \ \quad \quad \quad \quad \quad \quad \quad \quad \quad \quad \qquad \forall i \in \cB, d \in \cD_i \label{eqn:loadlogict}\\
		& y^\omega_i \geq y^\omega_g \ \quad \quad \quad \quad \quad \quad \quad \quad \quad \quad \qquad\ \forall i \in \cB, g \in \cG_i \label{eqn:genlogict}\\
		& y^\omega_i \geq y^\omega_{ij} \ \ \ \quad \quad \quad \quad \quad \quad \quad \quad \qquad\forall i \in \cB, (i,j) \in \cL \label{eqn:linelogict}\\
		& y^\omega_i \geq y^\omega_{ji} \ \ \ \quad \quad \quad \quad \quad \quad \quad \quad \qquad\forall i \in \cB, (j,i) \in \cL \label{eqn:linerevlogict}\\
		& y^\omega_c \leq z^\omega_c \ \quad\quad \quad \quad \quad \quad \quad \quad \quad \quad \quad \quad \quad \qquad\forall c \in \cC \label{eqn:operationlogict}\\
		& y^\omega_c \leq 1 - {\eta^\omega_{c}}\ \ \quad \quad \quad \quad \quad \quad \quad \quad \quad \quad \quad \qquad\forall c \in \cC \label{eqn:fireFunction} \\
		& {\eta^\omega_{c}} \geq v_c^\omega \ \quad \quad \quad \quad \quad \quad \quad \quad \quad \quad \quad \quad \quad \qquad\forall c \in \cC \label{eqn:exogenousFire} \\
		& {\eta^\omega_{k}} \geq u_c^\omega z^\omega_c \ \quad\quad \quad \quad \quad \quad \quad \quad \quad \qquad\forall c \in \cC, k \in I^\omega_{c} \label{eqn:ignitiont} \\
		& z^\omega_c = \hat{z}_{c \tau^\omega - 1} \ \qquad\quad\qquad\qquad \qquad\qquad \qquad \ \forall c \in \cC \label{eqn:nonanticipativity}\\
		& y^\omega_c, {\eta^\omega_{c}}, z^\omega_{c} \in \{0,1\}\ \ \quad\quad\quad\quad\quad\quad\quad\quad\quad\quad\ \forall c \in \cC. \label{eqn:binaryrestriction}
		\end{align}
	\end{subequations}
	The objective function includes the load-shedding cost after the disruption and the wildfire damage cost. The wildfire damage cost for a component \(c \in \cC\), denoted by \(r_c\), consists of the replacement cost of the electric components and the financial loss to the nearby communities. In model~\eqref{prob:staget}, the energization status will stay the same in the remaining time horizon, as we assume that all wildfire damages reveal at period \(\tau^\omega\) and no recovery decisions take place afterward. Constraints~\eqref{eqn:pfconslt}-\eqref{eqn:genlimitt} model an equivalent form of the DC power flow constraints as in model~\eqref{prob:stage1}. Constraints~\eqref{eqn:loadlogict}-\eqref{eqn:linerevlogict} indicate the functioning state of components, similar to their first-stage counterparts~\eqref{eqn:loadlogic1}-\eqref{eqn:linerevlogic1}. {We create a local copy of the first-stage shut-off decisions, \(z_c^\omega\), via the nonanticipativity constraint~\eqref{eqn:nonanticipativity}, which ensures that the shut-off state of each component at the time of wildfire disruption is identical for all scenarios that share the same history. The binding relationship is illustrated by Fig.~\ref{fig: scenario tree} and discussed in detail in~\cite{HaoxiangIJOC}.} With \(z_c^\omega\) indicating whether component \(c\) has been shut off, constraint~\eqref{eqn:operationlogict} makes sure that the shut-off components stay down through the second stage. Constraint~\eqref{eqn:fireFunction} states that damaged components no longer function, where we model the exogenous fire damage by constraint~\eqref{eqn:exogenousFire} and the endogenous fire damage by constraint~\eqref{eqn:ignitiont}. Notice that an endogenous fire started at component \(c\) requires both a fault occurrence \(u_c^\omega = 1\) and the component not being shut off \(z_c^\omega = 1\) and spreads to components \(k \in I_c^\omega\).

	\section{Decomposition Algorithm Based on Lagrangian Cut} \label{sec:algorithm}
	\noindent  The value function \(f^\omega\) is a non-convex function without an analytical form, as we need to solve a MIP subproblem~\eqref{prob:staget} to evaluate it. A common way to solve model~\eqref{prob:stage1} is to replace \(f^\omega\) with closed-form approximation, such as linear cutting planes. If the state variables are binary, we can generate Lagrangian cuts to equivalently reformulate \(f^\omega\) as a convex piecewise linear function { for each \(\omega \in \Omega\)}~\cite{zou2016nested}. Such Lagrangian cuts can be applied to model~\eqref{prob:stage1} as well because our state variables \(z\) are binary. In this section, we first show the derivation of the Lagrangian cut and the convergence properties of the cutting plane method. In addition, we propose an improved Lagrangian cut, \emph{square-minimization cut} (SMC), and illustrate how it can accelerate the cutting plane algorithm.
	
	At \(\ell\)-th iteration, we solve a Lagrangian dual problem~\eqref{prob:Lagrangian_Dual} to obtain the cut intercept \(v^{\omega, \ell}\) as its optimal value \(v^*\) and gradient \(\lambda^{\omega, \ell}\) as the optimal solution \(\lambda^*\). {It is worth noting that model \eqref{prob:Lagrangian_Dual} is a convex program and we can efficiently solve it using convex programming algorithms, such as the bundle method with stabilization techniques~\cite{lemarechal1995new}:}
	\begin{equation}
	v^* = \max_{\lambda} R^\omega(\hat{z},\lambda), \label{prob:Lagrangian_Dual}
	\end{equation}
	where the Lagrangian relaxation problem \(R^\omega(\hat{z},\lambda)\) is obtained by relaxing the nonanticipativity constraint~\eqref{eqn:nonanticipativity}: 
	\begin{align}
	R^\omega(\hat{z},\lambda) = \min \quad & \sum_{t = \tau^\omega}^T \sum_{d\in\mathcal{D}}w_d(1-x_{dt}^\omega) + \notag \\
	& \sum_{c\in\mathcal{C}} [r_c {\eta^\omega_{c}} + {\lambda_c (\hat{z}_{c, \tau^\omega - 1} - z_c^\omega)}]\\
	\st \quad & \text{Constraints }\eqref{eqn:pfconslt} - \eqref{eqn:ignitiont}\tcb{, \eqref{eqn:binaryrestriction}}. \notag
	\end{align}  
	
	Since we can always find a feasible solution to subproblem~\eqref{prob:staget} by setting all variables in model~\eqref{prob:staget} to zero, we do not need to generate feasibility cuts to characterize the domain of \(f^\omega\) and can focus on generating optimality cuts to approximate \(f^\omega\)~\cite{BirgLouv97}. For a specific $\omega$, we denote the lower approximation of $f^\omega$ by $V^\omega$, and write the $\ell$-th cut as follows:
	\begin{equation}
	V^\omega \geq v^{\omega, \ell} + (\lambda^{\omega, \ell})^\top (z_{\tau^\omega - 1} - \hat{z}^\ell_{\tau^\omega - 1}). \label{eqn:cuts}
	\end{equation}
	With construction of \tcb{\(L-1\)} Lagrangian cuts~\eqref{eqn:cuts}, we can substitute the function \(f^\omega\) by \(\) to obtain the lower approximation for model~\eqref{prob:stage1} as $(M_L)$. Since the cuts form a lower approximation of \(f^\omega\), the optimal value $Z^*_L$ is a lower bound of $Z^*$.
	\begin{align}
	(M_L) \quad Z_L^* =  \min\quad & \sum_{\omega\in\Omega} {p^\omega} \left[ \sum^{\tau^\omega - 1}_{t = 1} w_d(1 - x_{dt}) + V^\omega\right]  \\
	\mbox{s.t.}\quad & \text{Constraints }\eqref{eqn:pfconsl1} - \eqref{eqn:componenttime1} \notag\\
	& V^\omega \geq (\lambda^{\omega, \ell})^\top (z_{\cdot \tau^\omega - 1} - \hat{z}^\ell_{\cdot \tau^\omega - 1}) + \notag \\
	& \quad\quad v^{\omega, \ell}, \quad \forall \, \omega \in \Omega, \ell = 1, \dots, \tcb{L-1}. \notag
	\end{align}
	Once we obtain the optimal solution \((\hat{x},\hat{z},\hat{V},\hat{P})\) to \((M_L)\), it is a feasible solution to model~\eqref{prob:stage1} and we can calculate an upper bound by evaluating \(f^\omega(\hat{z}_{\cdot \tau^\omega - 1}, \xi^\omega)\) for each \(\omega \in \Omega\). We present the detailed steps of the decomposition algorithm in Algorithm~\ref{alg:decomposition} below, which iteratively updates the bounds. In the end, we terminate Algorithm~\ref{alg:decomposition} once the relative gap is within a predefined tolerance threshold \(\epsilon \geq 0\). We can show that Algorithm~\ref{alg:decomposition} converges to the optimal value in finite steps. 
	\begin{algorithm2e}
		\small
		\SetAlgoLined
		Initialization cut iteration number $\ell = \tcb{1}$, lower bound $LB = 0$, upper bound $UB = \infty$ and $\epsilon \geq 0$\; 
		\While{$\frac{UB-LB}{UB} \geq \epsilon$}{
			\tcc{Forward Steps}
			Solve the master problem $(M_\ell)$ and obtain the first-stage shut-off solution $\hat{z}^\ell$, optimal value $Z_\ell^*$, and the approximations of value function $\hat{V}^{\omega, \ell}$ \;
			Update $LB = Z_\ell^*$\;
			\For{$\omega \in \Omega$}{
				Solve model~\eqref{prob:staget} with \(\hat{z}^\ell\) and obtain $f^\omega(\hat{z}^\ell_{\cdot \tau^\omega}, \xi^\omega)$\;
			}
			Let $\bar{Z} = Z_\ell^*  + \sum_{\omega \in \Omega}p^\omega (f^\omega(\hat{z}^\ell_{\cdot \tau^\omega}, \xi^\omega) - \hat{V}^{\omega, \ell}$) \;
			\If{$\bar{Z} < UB$}{
				Update $UB = \bar{Z}$ and incumbent solution as $z^* = \hat{z}^\ell$\;
			}
			\tcc{Backward Steps}
			\For{$\omega \in \Omega$} {
				Solve the Lagrangian dual problem~\eqref{prob:Lagrangian_Dual}, and obtain the optimal solution $\lambda^{\omega, \ell}$ and the optimal value $v^{\omega, \ell} = R^\omega(\hat{z}^\ell, \lambda^{\omega, \ell})$\; 
				And add cut~\eqref{eqn:cuts} to $(M_\ell)$\;
			}
			
			
			Let $\ell = \ell + 1$\;
		}
		Output the $\epsilon$-optimal value $UB$ and solution $z^*$.
		\caption{Decomposition algorithm based on Lagrangian cuts to solve model~\eqref{prob:stage1}}\label{alg:decomposition}
	\end{algorithm2e}
	\begin{theorem}[Convergence] 
		When \(\epsilon = 0\), Algorithm~\ref{alg:decomposition} terminates in a finite number of iterations and outputs an optimal solution to model~\eqref{prob:stage1} after finitely many iterations $L$.
	\end{theorem}
	\begin{proof}
		{According to Theorem $3$ in Ref.~\cite{zou2016nested}, for a given incumbent first-stage solution $\hat{z}$ and scenario $\omega$, the Lagrangian cut generated at $\hat{z}$ is proven to be tight, meaning that its value at $\hat{z}$ is equal to the value function $f^\omega(\hat{z}_{\cdot, \tau^\omega - 1}, \xi^\omega)$ at this particular point. Since all state variables are binary, there is a finite number of distinct first-stage solutions $z$, allowing us to fully describe the value function using a finite number of Lagrangian cuts. Consequently, there exists a finite iteration number $k_1$, after which no new cuts are generated. For any iteration $k > k_1$, with the resulting first-stage state variable $\hat{z}^k$ and approximate second-stage value function values $\{V^{\omega, k}\}_{\omega \in \Omega}$, two cases may arise: 1) for each scenario $\omega$, $V^{\omega,k} = f^\omega(\hat{z}^k_{\cdot, \tau^\omega - 1}, \xi^\omega)$; in this scenario, the resulting upper bound \(\bar{Z}\) equals the optimal value of the master problem \(Z^*_k\), leading the algorithm to terminate with a global optimum; 2) there exists a scenario $\omega$ such that $V^\omega < f^\omega(\hat{z}^k_{\cdot, \tau^\omega - 1}, \xi^\omega)$; in this case, a Lagrangian cut is generated to enhance the approximation of the second-stage value function. However, this contradicts the result that no new cuts are generated after iteration $k_1$. Consequently, the algorithm converges to a global optimum in a finite number of steps.}
	\end{proof}
	
	Although the Lagrangian cut~\eqref{eqn:cuts} is tight and valid at the incumbent binary solution \(\hat{z}\) and results in finite convergence, the cut may be steep and not provide a good lower approximation at other solutions (shown as the black line in Fig.~\ref{fig: enhanced Lagrangian cut}). Therefore, We propose a rotated cut, the SMC, which is illustrated by the red line in Fig.~\ref{fig: enhanced Lagrangian cut}. Instead of solving model~\eqref{prob:Lagrangian_Dual} to obtain the cut coefficients \(\lambda\) and \(v\), we solve an alternative model:
	\begin{subequations}\label{prob: enhancement techniques}
		\begin{align}
		\min_{\lambda}\quad & \lambda^\top \lambda \\
		\mbox{s.t.}\quad & R^\omega(\hat{z},\lambda) \geq (1 - \delta)f^\omega(\hat{z}_{\cdot\tau^\omega -1 }, \xi^\omega), \label{eqn:smc_cons}
		\end{align}
	\end{subequations}
	and we let \(\lambda^{\omega,\ell}\) equal to its optimal solution \(\lambda^*\) and \(v^{\omega,\ell} = R^{\omega}(\hat{z},\lambda^*)\). Fig.~\ref{fig: enhanced Lagrangian cut} shows that the difference between a steep cut and a flat cut lies in the angle they make with the horizontal plane. We prefer a flat cut with a smaller \(\lambda^\top \lambda\), as \(\lambda\) represents the linear cut's coefficients. While we minimize \(\lambda^\top \lambda\) and rotate the cut, it is possible that the cut becomes non-tight at the incumbent solution \(\hat{z}\). Therefore, we set up constraint~\eqref{eqn:smc_cons} to force the cut value to be within a \(\delta\) neighborhood of \(f^\omega\) at \(\hat{z}\), which can be considered an ``anchor point." Model~\eqref{prob: enhancement techniques} is always feasible as the Lagrangian cut coefficients obtained from solving model~\eqref{prob:Lagrangian_Dual} can serve as a feasible solution. {As the function \(R^\omega\) is a concave function of \(\lambda\) given \(\hat{z}\), constraint~\eqref{eqn:smc_cons} characterizes a convex set and we can again use convex programming solution methods to solve model~\eqref{prob: enhancement techniques}.}
	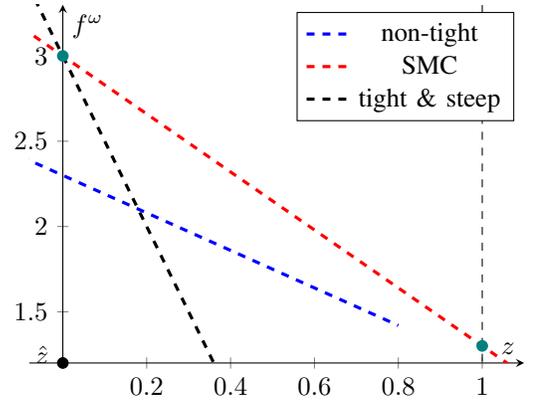
\begin{figure}[h]
		\begin{center}
			\begin{tikzpicture}
			\begin{axis}[
			axis lines = middle,
			xlabel = {$z$},
			ylabel = {$f^\omega$},
			xmin = -.08, xmax = 1.1,
			ymin = 1.2, ymax = 3.3,
			major grid style = {lightgray},
			minor grid style = {lightgray!25},
			width = .45\textwidth,
			height = 0.35\textwidth]
			\addplot[name path = Cut1, 
			domain = -0.1:.8,
			dashed,
			very thick,
			blue,]{ 2.3 - 1.1 * x};
			\addplot[name path = Cut2, 
			domain = -0.1:1.1,
			dashed,
			very thick,
			red,]{ 3 - 1.7 * x};
			\addplot[name path = Cut3, 
			domain = -0.1:.8,
			dashed,
			very thick,
			black,]{ 3.0 - 5 * x};
			\addplot[mark=*,only marks, black] coordinates {(0, 1.2) };
			\addplot[mark=*,only marks, teal] coordinates {(0,3) (1,1.3)};
			\node[black] at (-.05,1.25) {$\hat{z}$};
			\draw [dashed] (axis cs:{1.0,0}) -- (axis cs:{1.0,3.5});
			\legend{non-tight , SMC, tight \& steep}
			\end{axis}
			\end{tikzpicture}
		\end{center}
		\caption{Illustration of different types of cuts: the black cut and the red cut are tight at \(\hat{z} = 0\), but the blue cut is not since its value at $\hat{z}$ is smaller than the function value; the red cut has better performance than the black cut as it also characterizes the function value exactly at another \(\hat{z} = 1\).}\label{fig: enhanced Lagrangian cut}
	\end{figure}

	\section{Numerical Results} \label{sec:results}
	\noindent In this section, we first describe the numerical experiment setup, with which we evaluate the efficiency of our decomposition method and the solution quality. We then compare our two-stage SMIP model~\eqref{prob:stage1} with \tcb{four} benchmark cases: a deterministic model that ignores wildfire disruption, a wait-and-see model with perfect wildfire information, a model based on an aggregate wildfire risk metric, \tcb{and a robust optimization model that seeks to minimize the total cost of the worst-case scenario.} Finally, we examine the robustness of our nominal solution against potential inaccuracy in wildfire disruption probability. 
	
	\subsection{Experiment Setup}
	We use the RTS-GMLC $73$-bus case~\cite{RTS-GLMC} and run experiments on a $24$-hour horizon with each time period accounting for an hour (\(T = 24\)). We assume a constant electricity demand except for the peak hours, $9$ am to $12$ pm and $3$ pm to $7$ pm, during which the demand is $1.2$ times the non-peak value. 
	
	The economic loss caused by fire and sudden power outages depends on the fire intensity and the amount of load in the affected area. To quantify the economic impact of wildfire damage and load loss, we assign each electrical component a numerical rating based on its importance and impact on the area, following a similar approach to the \textit{Value Response Index} in Texas Wildfire Risk Rating System~\cite{TWRAS}. 
	{We assign each electrical component a scaled cost based on value data from \cite{NuclearPower}, \cite{WindPower}, \cite{NuclearFinancing}, and \cite{TransmissionCost} to create comparable cost values for equipment damage and shed load. We obtain the following cost parameters: }  
	\begin{enumerate} 
		\item Load priorities \(w_d\) range from $50$ to $1000$; 
		\item The damage costs \(r_c\) of wind turbines, thermal and nuclear power plants are $50$, $1000$, and $2500$, respectively; 
		\item The damage cost \(r_c\) of each bus is $50$; 
		\item The damage cost \(r_c\) of a transmission line is $0.285 \ell$, where $\ell$ is the length of the transmission line. 
	\end{enumerate}
	Fig.~\ref{fig: deenergize_18} color codes the load and generator priority levels based on their load-shedding and damage costs, respectively. 
	\begin{figure}
		\centering
		\setkeys{Gin}{width=.85\linewidth}
		\includegraphics{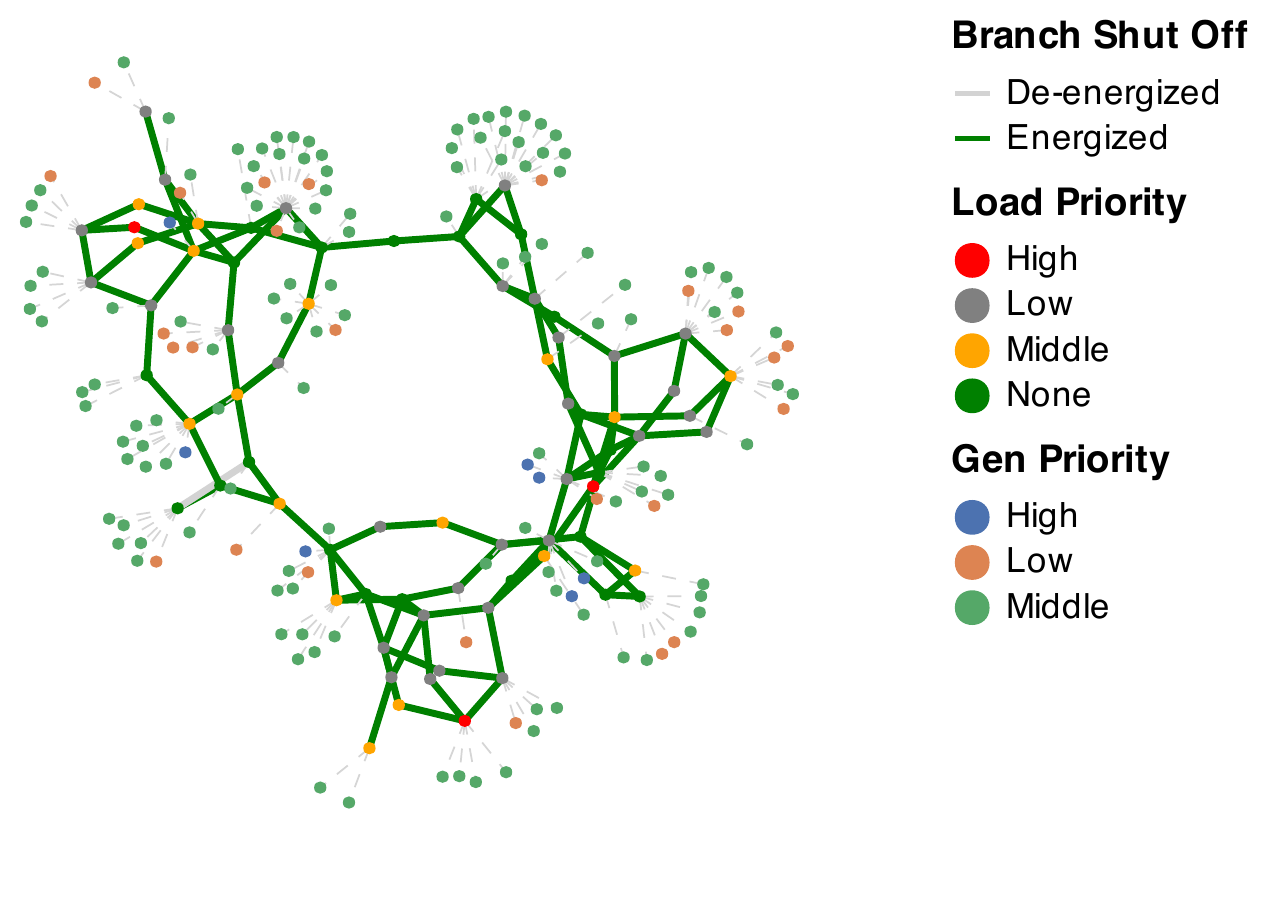}
		\caption{Illustration of the RTS-GMLC system.}
		\label{fig: deenergize_18}
	\end{figure}
	
	{All optimization models were implemented using JuMP~\cite{JuMP} in Julia v1.6~\cite{Julia} and solved by Gurobi v9.5.1 \cite{gurobi} on a computer with a 10-core M1 Pro CPU and $32$GB memory. The network plots are generated using PowerPlot.jl, which depends on PowerModels.jl package \cite{PowerModels}. Scenario simulation is constructed by an agent-based model package, Agent.jl \cite{Agents}. For Algorithm~\ref{alg:decomposition} and SMC, we set $\epsilon = 1\%$ and $\delta = 10^{-4}$.}
	
	\subsection{Computational Performance and Solution Analysis} \label{subsec:computational}
	We examine how the algorithm performance and solution quality change with sample size. 
	Our two-stage SMIP model, utilizing a set of scenarios \(\Omega\), can be considered sample average approximation (SAA) and serve as a lower bound estimator of the optimal value for the true stochastic process of wildfire progress. As the number of scenarios increases, the approximation precision improves, but the required computational effort also increases~\cite{BirgLouv97}. Our objective is to determine a sample size that yields a high-quality solution within an acceptable computational budget. To achieve this, we first obtain twenty first-stage SAA solutions \(\hat{X} = \{\hat{x},\hat{z},\hat{\theta},\hat{P}\}\) and optimal values with sample sizes of $20, 50, 100, 200$, and $500$.  
	
	Next, we generate \(n = 5,000\) wildfire scenarios \(\tilde{\Omega}\) using the procedure outlined in Section~\ref{subsec:scenario} for the out-of-sample test. We evaluate the expected total cost of each SAA solution using those $5,000$ scenarios as follows:
	
	\vspace{-0.3cm}
	
	{\small
		\begin{align}
		& g_n(\hat{X}) =  \notag \\
		& \min \ \sum_{\omega \in \tilde{\Omega}} \frac{1}{n} \left[ \sum_{t = 1}^{\tau^\omega-1} \sum_{d \in \cD} w_d (1-\hat{x}_{dt}) + f^\omega(\hat{z}_{\cdot \tau^\omega-1}, \xi^\omega) \right]. \label{eqn:obj}
		\end{align}
	}
	Such cost can serve as an upper-bound estimator for the true wildfire progress. We take the mean of those twenty lower- and upper-bound estimators and calculate their $95\%$ confidence intervals for each sample size, shown in Fig.~\ref{fig: CI}. As the figure illustrates, the gap between the lower- and upper-bound estimators narrows, and both estimators become less variable with increasing sample size. We consider this gap acceptable for a sample size of 500. 
	
	{As the sample size increases, solving the SMIP model \eqref{prob:stage1} becomes increasingly challenging. Table~\ref{table: time comparison} shows the effectiveness of Algorithm \ref{alg:decomposition} with SMC. With a sample size of $500$, this algorithm is capable of solving the model within a reasonable time frame ($18,000$ seconds). In contrast, state-of-the-art solvers like Gurobi are unable to handle such a problem. Table~\ref{table: time comparison} highlights the necessity of computational improvement through the use of SMC, as the algorithm with the original Lagrangian cuts (LC) fails to reach the optimal solution within the given time constraint. }
	
	{We simulate scenarios and solve the two-stage SMIP offline in advance of a specific day of operations, which provides a nominal plan. During the day, we execute the nominal plan until a wildfire disruption occurs. \tcb{Given the wildfire disruption realization}, we can solve the second-stage problem to obtain recourse decisions, a process that typically takes less than a second. Therefore, the run-time of Algorithm~\ref{alg:decomposition} is sufficient for the practical implementation of our model.}
	
	\begin{figure}
		\centering
		\setkeys{Gin}{width=.85\linewidth}
		\includegraphics{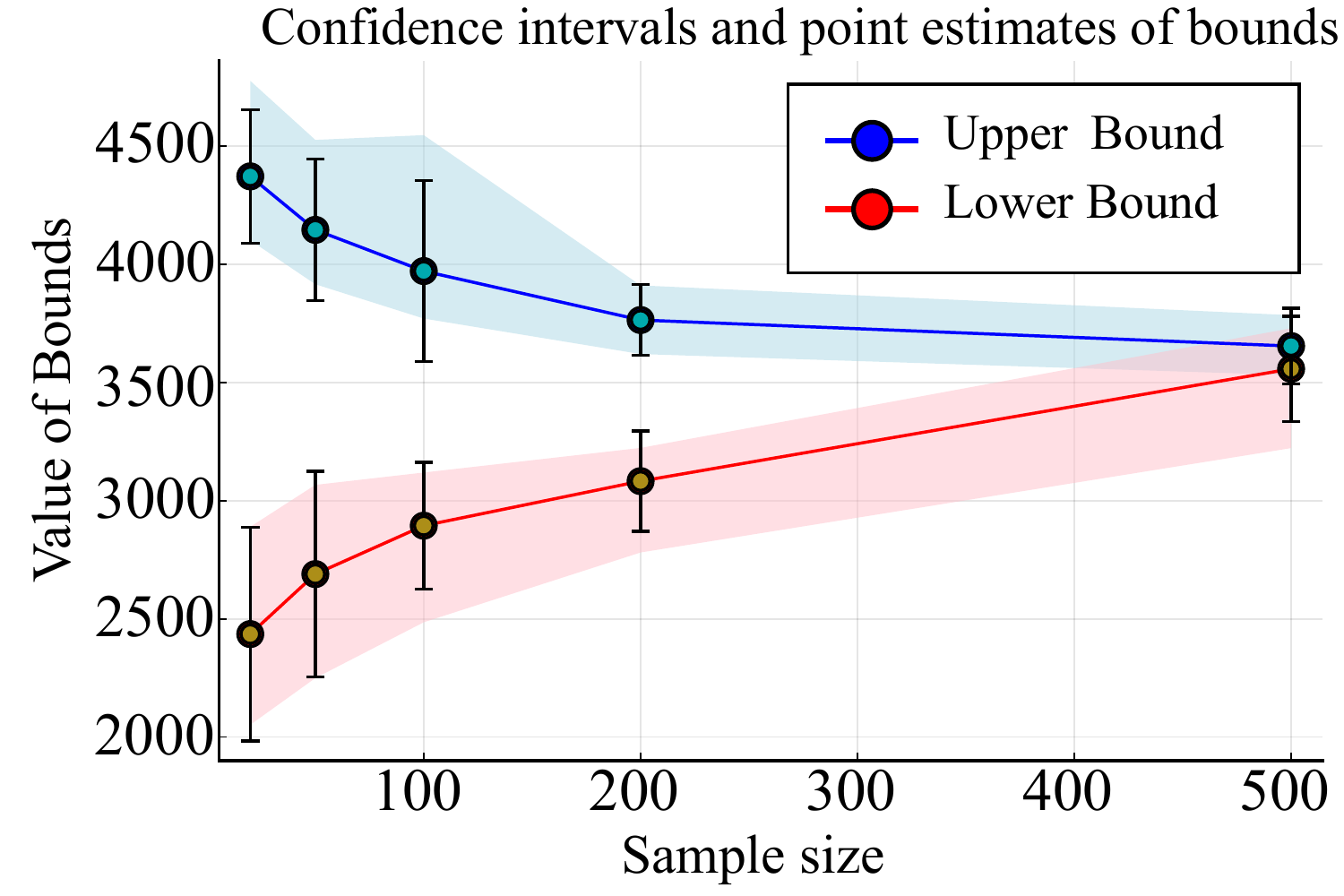}
		\caption{Confidence intervals (bar) and point estimates (circle) of the lower and upper bounds for solutions with different sample sizes. The shaded area indicates the range of extreme values (maximum and minimum).}
		\label{fig: CI}
	\end{figure}
	
	\begin{table}[h]
		\centering
		\caption{Run-time (sec.) of Algorithm~\ref{alg:decomposition} utilizing SMC and LC with a tolerance of $1.0\%$, compared with a run-time of solving model~\eqref{prob:stage1} by Gurobi with MIPGap $= 1\%$.}
		\begin{tabular}{c|cc|c}\toprule[1pt]\hline
			\multirow{2}{*}{Sample size} & \multicolumn{2}{c|}{Algorithm~\ref{alg:decomposition} Time}  & Gurobi Time  \\ 
			& SMC (Sec.)              & LC (Sec.)                & (Sec.)  \\ 
			\cline{1-4} 
			$20$        & $712$             & $3287$             & $84$                      \\
			$50$        & $1248$            & $8741$             & $376$                     \\
			$100$       & $2604$            & $17911$            & $1510$                    \\
			$200$       & $5323$            & $>18000$           & $9186$                    \\
			$500$       & $12776$           & $>18000$           & $>18000$                  \\ \hline  
			\bottomrule[1pt]    
		\end{tabular}
		\label{table: time comparison}
	\end{table}

	The nominal plan obtained by model~\eqref{prob:stage1} balances the trade-off between minimizing load-shedding cost and reducing wildfire risk. For example, transmission line $72$ has a high risk of starting an endogenous wildfire, as it ranks {$10$-th} among all transmission lines regarding the number of scenarios where {an endogenous wildfire occurs. However, in a $24$ time-period horizon, the nominal plan does not de-energize it early }because it connects to a bus with multiple generators that supply a large amount of load; see Fig.~\ref{fig: nominal_pan_12}. To illustrate this trade-off, we compare the outcomes of de-energizing and keeping line $72$ energized. If we \textit{intentionally} de-energize line $72$ early at \(t = 12\), it causes a significant loss of generation and load-shed; see Fig.~\ref{fig: nominal_plan_controlled}. The load-shedding cost outweighs the expected damage cost caused by line $72$, so the nominal plan prefers to keep it energized. However, after period $22$, the nominal plan de-energizes it (Fig.~\ref{fig: nominal_plan_22}) because the expected damage cost from its fault exceeds the load-shedding cost from less generation. {We observe a similar ``end-of-horizon" effect when we extend the time horizon length to \(T = 36\) and \(T = 48\).}
	
	\begin{figure*}
		\centering
		\setkeys{Gin}{width=\linewidth}
		\begin{subfigure}{0.32\textwidth}
			\includegraphics{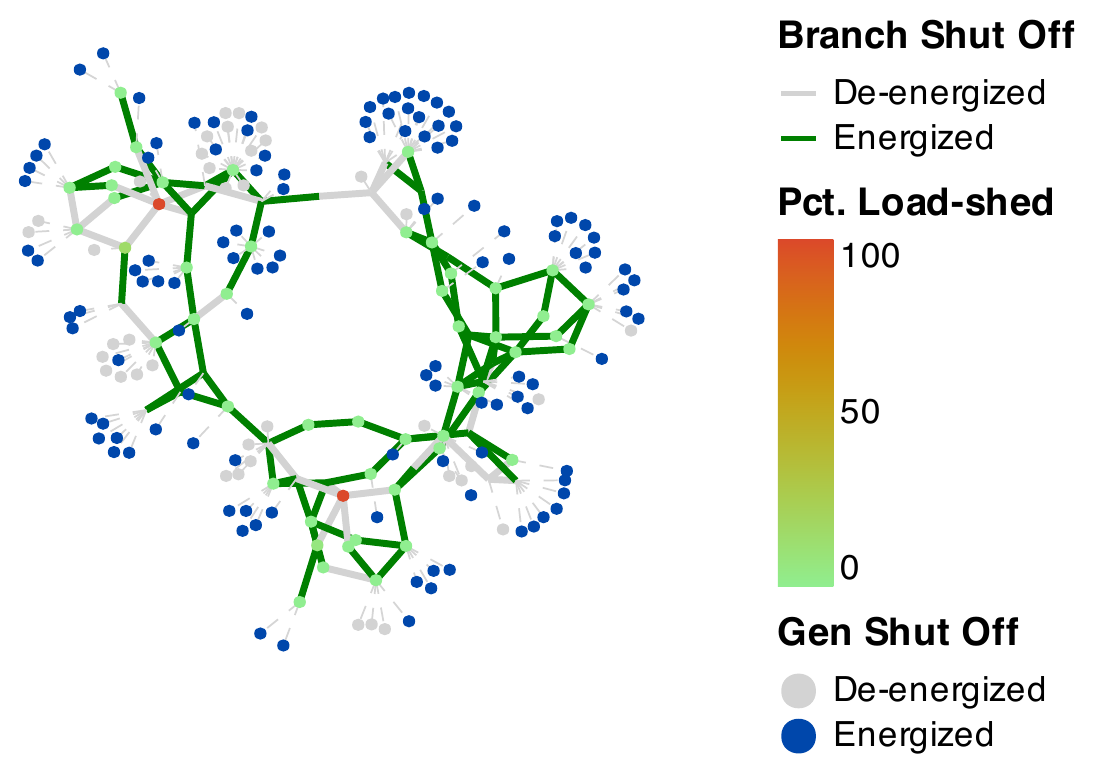}
			\caption{The nominal load plan for the period $12$ and its associated load-shedding percentage illustrates how the nominal plan optimizes the load distribution and minimizes the risk exposure under uncertainty.}
			\label{fig: nominal_pan_12}
		\end{subfigure}
		\hfil
		\begin{subfigure}{0.32\textwidth}
			\includegraphics{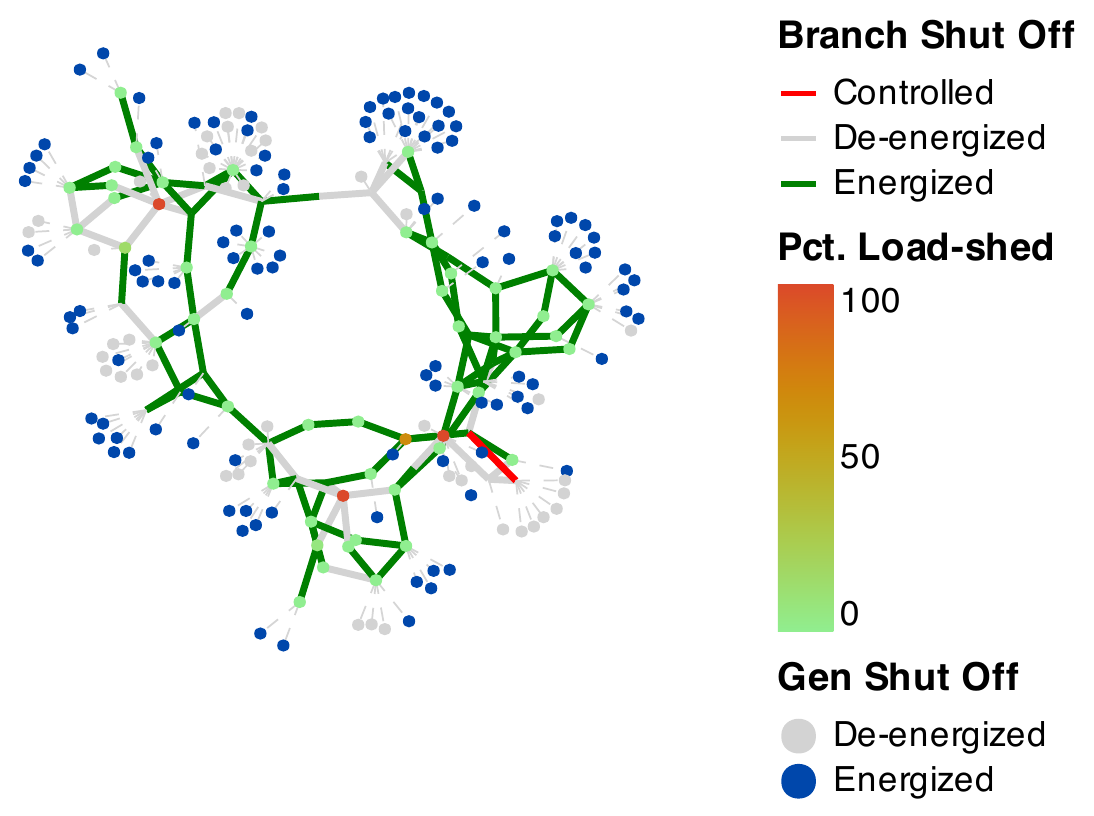}
			\caption{The additional load shed in period $12$ when transmission line $72$ is de-energized intentionally, compared to when the nominal plan keeps it on. }
			\label{fig: nominal_plan_controlled}
		\end{subfigure}
		\hfil
		\begin{subfigure}{0.32\textwidth}
			\includegraphics{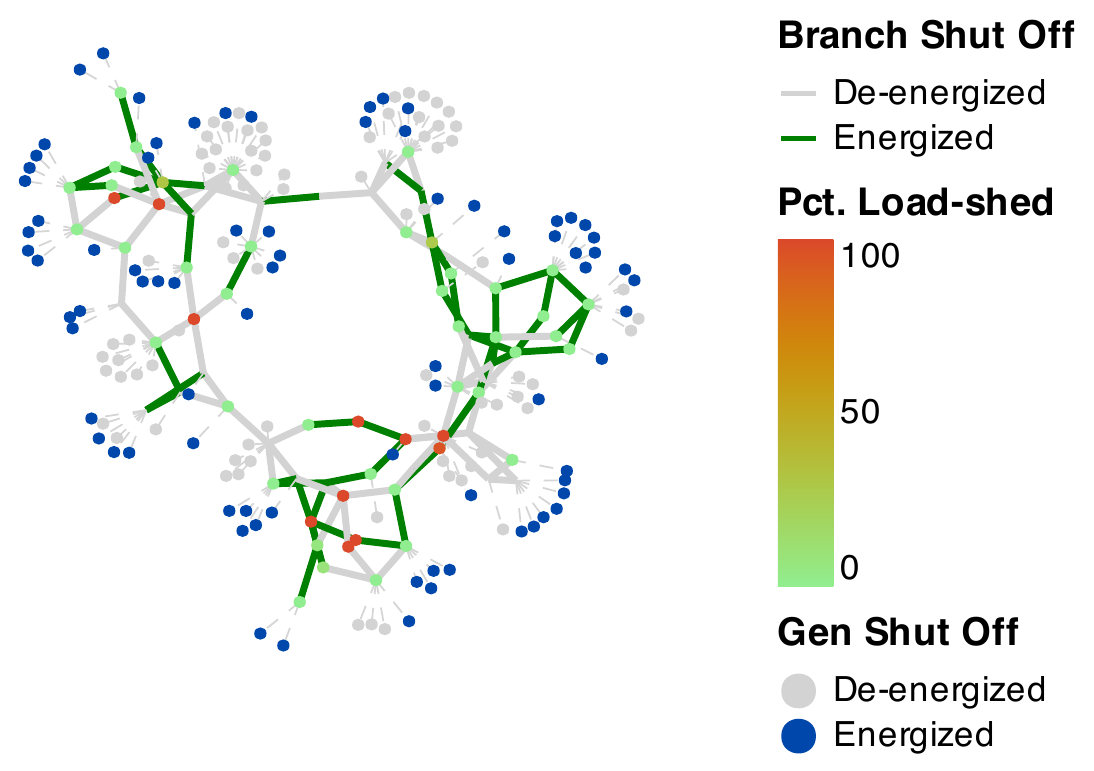}
			\caption{The nominal load plan for period $22$, illustrates how the nominal plan optimizes the load distribution and minimizes the risk exposure under uncertainty.}
			\label{fig: nominal_plan_22}
		\end{subfigure}
		\caption{Visualization of trade-off in the nominal plan. The red line in the bottom-right corner represents transmission line $72$.}
		\label{fig: nomial_plan_tradeoff}
	\end{figure*}
	
	\subsection{Benchmark Comparison} \label{subsec:benchmark}
	We choose the best optimal solution obtained from twenty batches with a sample size of $500$ in Section~\ref{subsec:computational} as the solution to our two-stage SMIP model~\eqref{prob:stage1}, denoted as \(X^* = \{x^*,z^*,\theta^*, P^*\}\). We then compare this optimal solution, \(X^*\), to the following three benchmark models: 
	\begin{enumerate}
		\item A deterministic model where we maximize load satisfaction without considering possible disruption:
		\begin{subequations}
			\label{prob:benchmark1 stage1}
			\begin{align}
			(M^{det}) \;\; \min \quad & \sum_{t = 1}^{T} \sum_{d \in \cD} w_d (1-x_{dt}) \span \span \notag \\
			\st \quad & \text{Constraints }\eqref{eqn:pfconsl1} - \eqref{eqn:componenttime1}.  \notag
			\end{align}
		\end{subequations}
		Since $(M^{det})$ does not consider the wildfire risk, the optimal solution will return a nominal plan \(X^{det} = \{x^{det},z^{det},\theta^{det}, P^{det}\}\) with no shut-off.
		
		\item A wait-and-see model where we can make a scenario-specific plan assuming $\xi^\omega$ is known before we make decisions. For each \(\omega \in \Omega\), we solve
		\begin{subequations}
			\label{prob:benchmark2 oracle}
			\begin{align}
			(M^{ws,\omega}) \qquad \span g^{ws,\omega} = \qquad \qquad \qquad \qquad  \notag\\
			\min \quad & \sum_{t = 1}^{\tau^\omega-1} \sum_{d \in \cD} w_d (1-x_{dt}) + f^\omega(z_{\cdot \tau^\omega-1}, \xi^\omega) \notag  \\
			\st \quad & \text{Constraints }\eqref{eqn:pfconsl1} - \eqref{eqn:componenttime1},  \notag
			\end{align}
		\end{subequations}
		and obtain a scenario-specific solution \(X^{ws,\omega} = \{x^{ws,\omega}, z^{ws,\omega}, \theta^{ws,\omega}, P^{ws,\omega}\}\). An expected wait-and-see cost can then be computed as $g_{n}^* = \sum_{\omega \in \Omega}{\frac{1}{n}} g^{ws,\omega}$.
		
		\item A wildfire-risk-based model proposed by \cite{Noah2022} to simultaneously reduce wildfire risk and {load-shedding}:
		\begin{subequations}
			\label{prob:benchmark3}
			\begin{align}
			(M^{rb}) \;\; \min \quad & \sum_{t = 1}^T \left[ \dfrac{\alpha R_{Fire, t}}{R_{Tot}} - (1-\alpha) \dfrac{\sum_{d \in \cD}x_{dt}w_dD_{dt}}{D_{Tot}} \right] \notag \\
			\st \quad & \mbox{Constraints }\eqref{eqn:pfconsl1} - \eqref{eqn:componenttime1}  \notag\\
			& R_{Fire,t}=\sum_{c \in \cC} R_{ct} z_{ct} \quad\quad \forall t \in \cT, \notag
			\end{align}
		\end{subequations}
		The component-wise risk measure $R_{ct}$ is the mean number of components affected by an endogenous fire started at component \(c\) across all scenarios. Fig.~\ref{fig:risk_value} illustrates the value of \(R_{ct}\) for $t = 13$. The components that face significant wildfire risk are marked in the figure. By controlling the number of energized components, we can compute the total wildfire risk $R_{Fire, t}$ for the period $t$. We use the total load $D_{Tot} = \sum_{d\in\cD, t\in\cT}D_{dt}$ and total wildfire risk $R_{Tot} = \sum_{c\in\cC, t\in\cT}R_{ct}$ as the normalization factors such that we can combine the wildfire risk and the {load-shedding} in the objective function. The parameter $\alpha \in [0, 1]$ adjusts the trade-off between serving loads and reducing wildfire risk. Solving model \((M^{rb})\) returns a nominal plan \(X^{rb,\alpha} = \{x^{rb,\alpha}, z^{rb,\alpha}, \theta^{rb,\alpha}, P^{rb,\alpha}\}\). 
		\begin{figure}
			\centering
			\setkeys{Gin}{width=.7\linewidth}
			\includegraphics{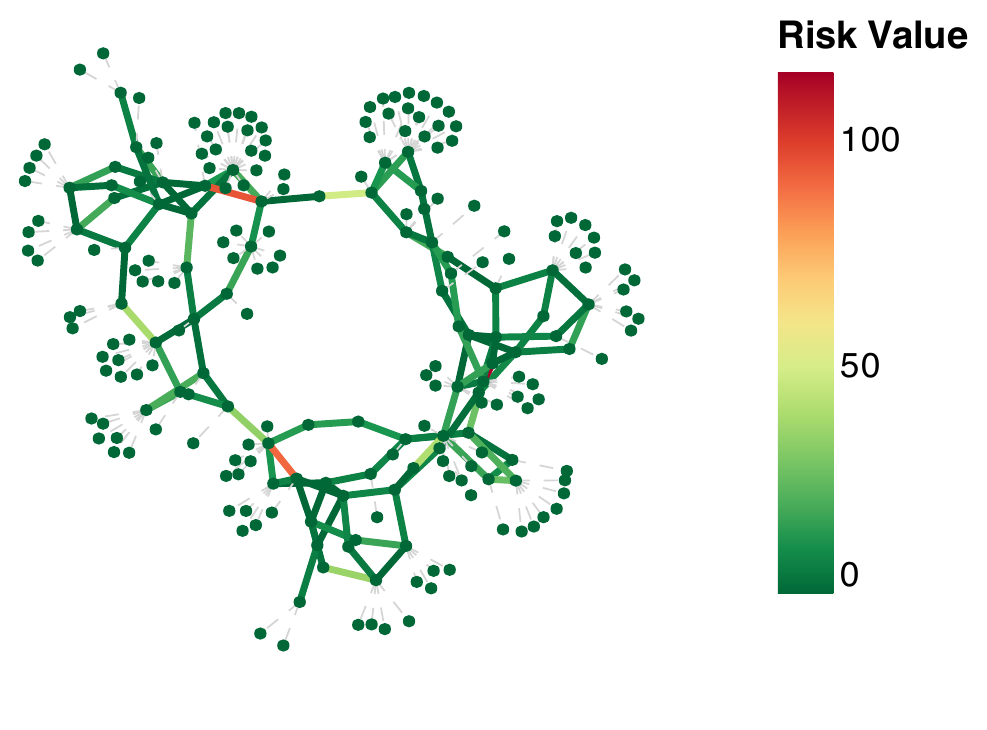}
			\caption{Illustration of wildfire risk \(R_{ct}\) at $t = 13$. }
			\label{fig:risk_value}
		\end{figure}
		
		\tcb{\item A scenario-based robust optimization model, in which we minimize the total cost of the worst-case scenario:}
		
		\vspace{-0.3cm}
		
		\tcb{\small
			\begin{subequations}
				\label{prob:benchmark4}
				\begin{align}
				(M^{ro})\ \min & \ \max_{\omega \in \Omega} \ \sum_{t = 1}^{\tau^\omega-1} \sum_{d \in \cD} w_d (1-x_{dt}) + f^\omega(z_{\cdot \tau^\omega-1}, \xi^\omega) \notag \\
				\st & \ \eqref{eqn:pfconsl1}-~\eqref{eqn:binaryrestriction1} \qquad\qquad\qquad\qquad\qquad \forall t \in \cT \notag.
				\end{align}
		\end{subequations}}
		
		\vspace{-0.3cm}
		
		\tcb{Solving model $(M^{ro})$ returns a nominal plan \(X^{ro} = \{x^{ro},z^{ro},\theta^{ro}, P^{ro}\}\).}
	\end{enumerate}
	
	Each model is solved, then evaluated with the same $n = 5,000$ out-of-sample scenarios \(\tilde{\Omega}\) used to evaluate the two-stage SMIP model~\eqref{prob:stage1}. For the deterministic model and the wait-and-see model, we obtain two following ratios to show their relative difference compared to model~\eqref{prob:stage1}:
	\begin{align*}
	\dfrac{g_{n}(X^{det}) - g_{n}(X^*)}{g_{n}(X^{*})} = 0.639  \mbox{ and }  \dfrac{g_{n}(X^*) - g_{n}^*}{g_{n}(X^*)} = 0.416.
	\end{align*}
	The first ratio emphasizes that the two-stage SMIP can significantly improve the operations compared to the deterministic approach agnostic to potential wildfires. The second ratio quantifies the value of perfect information. We can save $41.6\%$ of the cost if we can forecast the wildfire perfectly before we make the shut-off decisions. This large ratio reflects the high variance in wildfire scenarios and the difficulty for a single nominal shut-off plan to perform optimally in all scenarios. While it is not entirely possible to obtain the perfect information, it is helpful if we can obtain auxiliary data, conditioned on which we can identify a better representation of the scenarios. For example, we may exclude the scenarios with wildfires occurring in areas with forecast precipitation. 
	
	Fig.~\ref{fig:benchmark} illustrates such effect by selecting a subset of twenty scenarios: we can observe that the cost of the SMIP solution is significantly lower than the cost of the deterministic solution in all but one of the selected scenarios and can approach the best possible cost in two scenarios. If the auxiliary data happen to show those two scenarios are likely to happen, our SMIP solution will perform well. On the contrary, if the auxiliary data tell the other way, we may want to re-solve the two-stage SMIP model~\ref{prob:stage1} with a refined set of scenarios.
	
	\begin{figure}
		\centering
		\includegraphics[scale=0.40]{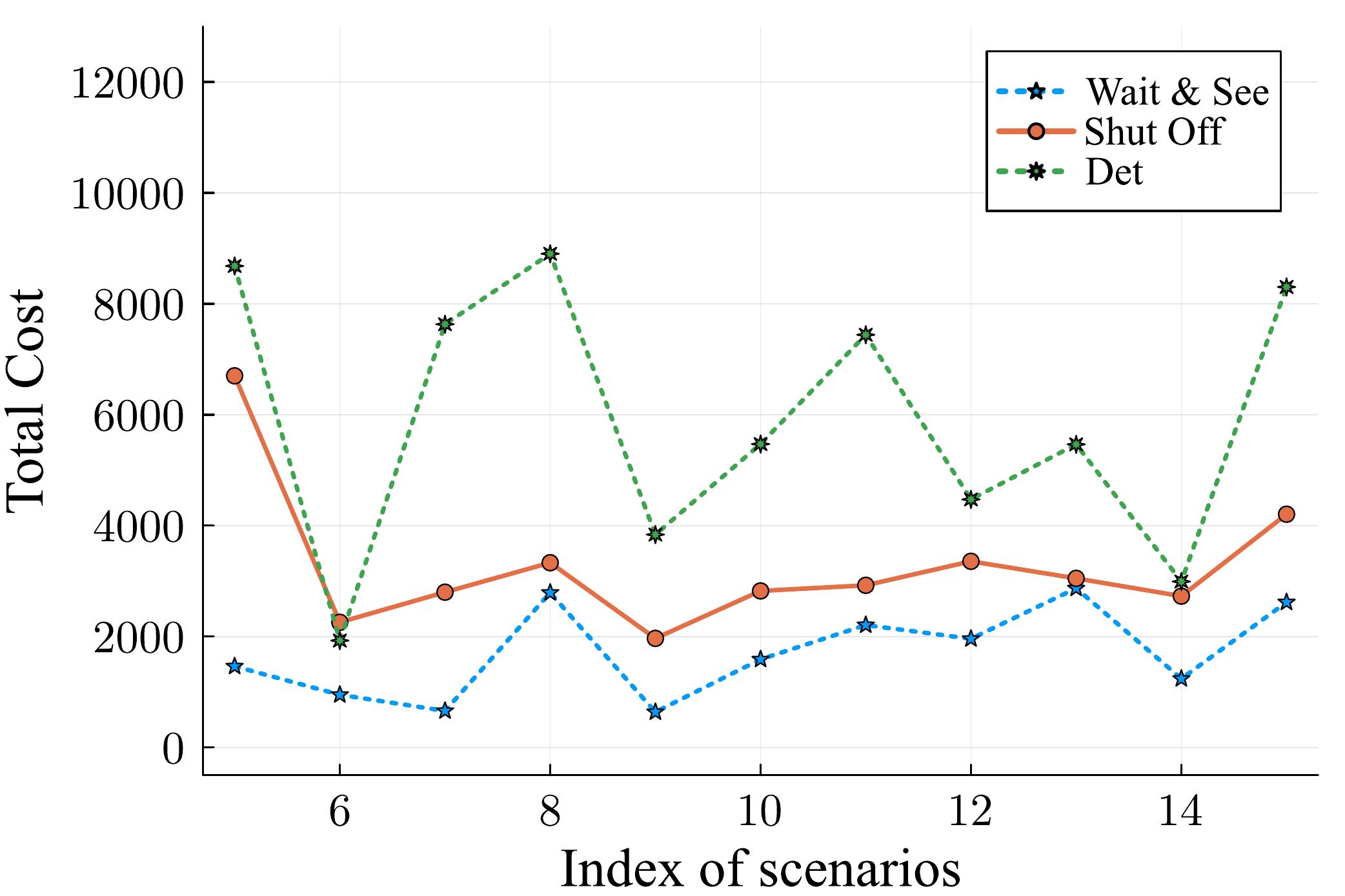}
		\caption{Scenario-specific costs for model~\eqref{prob:stage1}, the deterministic model $(M^{det})$, and the wait-and-see model $(M^{ws,\omega})$. For each scenario, the green star and orange circle represent the costs incurred by the nominal plans generated by $(M^{det})$ and model~\eqref{prob:stage1}. The blue star marks the wait-and-see cost.}
		\label{fig:benchmark}
	\end{figure}
	
	For the wildfire-risk-based model $(M^{rb})$, we use the set of scenarios $\Omega^*$, with which we obtain $X^*$, to calculate the wildfire risk value $R_{Fire, t}$ for all $t \in \cT$. We obtain a solution ${X}^{rb,\alpha}$ for $\alpha \in \{0.0, 0.1, \dots, 0.9\}$. We list the out-of-sample mean load-shedding cost and wildfire damage cost under non-disruptive scenarios (no wildfire ignitions) and disruptive scenarios in Table~\ref{table: cost for different alpha}, evaluated with \(\tilde{\Omega}\). 
	
	If the wildfire risk is ignored ($\alpha = 0.0$), it is equivalent to the deterministic solution \(X^{det}\), which leads to a large out-of-sample cost \(g_n(X^{det})\). By increasing $\alpha$, we can first reduce the total cost via the preventive shut-offs of potentially risky components, resulting in less damage cost. However, as $\alpha$ increases further, the nominal plan becomes too conservative, and many components are de-energized such that the increasing load-shedding cost outweighs the reduced damage costs. This illustrates a major challenge in the wildfire-risk-based model $(M^{rb})$ as there is little guidance to select a proper $\alpha$. 
	
	We then compare our SMIP model~\eqref{prob:stage1} with the wildfire-risk-based model $(M^{rb})$, via the ratio of relative improvements (RRI) $(g_n({X}^{rb,\alpha}) - g_n(X^*))/g_n(X^*)$ in Table~\ref{table: cost for different alpha}. The results show that the nominal plan obtained by solving model~\eqref{prob:stage1} performs significantly better{, and the RRI values for $X^{rb, \alpha}$ exceed $45\%$ for all \(\alpha\) selected. This indicates that, without considering wildfire scenarios and potential adaptive recourse actions, static non-robust planning is far from optimum, regardless of the level of conservatism.} The reason is that $(M^{rb})$ only considers minimizing the total wildfire risk and satisfying the power demand, which leads to the nominal plans that de-energize any component with $R_{ct}>0$ for some $t \in \cT$ as early as possible, only if this action does not affect the future supply of electricity. However, this property may compromise the resilience of the power grid by de-energizing an excessive number of components. If a disruption occurs, there might be very high load-shedding costs. 
	
	\begin{table}[ht]
		\centering
		\caption{Load-shedding costs and damage costs of nominal plans $X^{rb,\alpha}$ under non-disruptive and disruptive scenarios. }
		\label{table: cost for different alpha}
		\begin{tabular}{c|c|cc|cc}\toprule[1pt]\hline
			& \multicolumn{1}{c|}{Non-Disr.}  & \multicolumn{2}{c|}{Disruptive} &    \multicolumn{2}{c}{{Total Cost}}          \\ 
			$\alpha$& Load shed  & Load shed & \multicolumn{1}{c|}{Damage}       &    $g_n(\cdot)$    &   RRI  \\  \cmidrule{1-6} 
			$0.0/X^{det}$   & $0.0$                     & $3247.0$               & $2814.0$   & $5936.2$  & $63.9\%$   \\
			$0.1$   & $148.6$                   & $3843.5$               & $1767.5$   & $5495.4$  & $51.7\%$   \\
			$0.2$   & $162.9$                   & $3850.7$               & $1739.8$   & $5480.8$  & $51.4\%$   \\
			$0.3$   & $221.4$                   & $3891.0$               & $1713.9$   & $5489.4$  & $51.6\%$   \\
			$0.4$   & $171.6$                   & $3786.3$               & $1783.4$   & $5454.9$  & $50.6\%$   \\
			$0.5$   & $251.1$                   & $3712.8$               & $1773.9$   & $5373.7$  & $48.4\%$   \\
			$0.6$   & $106.7$                   & $3686.8$               & $1817.0$   & $5390.5$  & $48.9\%$   \\ 
			$0.7$   & $469.3$                   & $3721.4$               & $1754.1$   & $5425.9$  & $50.0\%$   \\ 
			$0.8$   & $883.6$                   & $3949.0$               & $1700.1$   & $5532.7$  & $52.8\%$   \\ 
			$0.9$   & $1505.8$                  & $4285.3$               & $1341.1$   & $5510.5$  & $52.2\%$   \\ \cmidrule{1-6}
			$X^*$ & $1678.2$                   & $2355.2$              & $1333.6$      & $3612.7$ & $0.0\%$      \\ \hline
			\bottomrule[1pt]    
		\end{tabular}
	\end{table}
	
	\tcb{We use a scenario-appending algorithm similar to that in Ref.~\cite{yang2021robust}. The algorithm starts with solving a relaxed model considering only a prespecified subset of scenarios $\hat{\Omega} \subset \Omega^*$ and obtains a solution \(\hat{X}\). An oracle is run to find the worst-case scenario \(\hat{\omega} \in \Omega^*\) given \(\hat{X}\). We append such a scenario to the set \(\hat{\Omega}\) and repeat the above process until the obtained \(\hat{\omega}\) is already in \(\hat{\Omega}\). The experiment demonstrates that the subset $\hat{\Omega}$ includes $23$ scenarios, much smaller than \(|\Omega^*| = 500\), allowing us to run the algorithm in $109$ seconds. Although the robust optimization algorithm runs faster, the robust solution $X^{ro}$ presents worse out-of-sample test results using scenario set $\tilde{\Omega}$ in TABLE~\ref{table:robustbenchmark}, compared with those of the SMIP solution \(X^*\). The robust nominal plan $X^{ro}$ de-energizes more components early in the time horizon to mitigate endogenous wildfire risks, which is illustrated in Fig.~\ref{fig:robustbenchmark}. Consequently, this conservative strategy results in a substantial amount of load-shedding in both nominal and disruptive scenarios, leading to a total cost of up to four times that of $X^*$. The worst-case in-sample cost among scenarios in $\Omega^*$ for $X^{ro}$ is $32,374$ compared to $34,016$ for $X^*$, which is reasonable as $X^{ro}$ optimizes the worst-case scenario. However, the out-of-sample worst-case costs, as demonstrated in the last column of TABLE~\ref{table:robustbenchmark}, indicate that \(X^{ro}\) is sensitive to the scenario selection and performs poorly when facing new samples. In contrast, our stochastic programming solution $X^*$ shows robustness as the out-of-sample worst-case cost is even smaller than the in-sample one. Moreover, the performance of $X^{ro}$ significantly lags behind that of $X^*$ across other scenarios. Although the robust model requires less time to generate a feasible nominal plan, the numerical results show that the robust nominal plan is conservative because it relies on limited scenario information and prioritizes minimizing worst-case costs.}

	\begin{table}[ht]
		\centering
		\caption{\tcb{Load-shedding costs and damage costs of nominal plans $X^{ro}$ under non-disruptive and disruptive scenarios. }}
		\label{table:robustbenchmark}
		\begin{tabular}{c|c|cc|cc}\toprule[1pt]\hline
			& \tcb{Non-Disr.}  & \multicolumn{2}{c|}{\tcb{Disruptive}} &    \multicolumn{2}{c}{{\tcb{Cost}} }         \\ 
			& \tcb{Load shed}         & \tcb{Load shed}     & \multicolumn{1}{c|}{\tcb{Damage}}        & \tcb{$g_n(\cdot)$}      & \tcb{Worst-case}  \\ \cmidrule{1-6} 
			\tcb{$X^{ro}$}   & \tcb{$30130.2$}         & \tcb{$18067.3$}     & \tcb{$2506.3$}                           & \tcb{$20770.5$}         & \tcb{$56848.3$}   \\ \cmidrule{1-6} 
			\tcb{$X^*$}      & \tcb{$1678.2$}          & \tcb{$2355.2$}      & \tcb{$1333.6$}                           & \tcb{$3612.7$}          & \tcb{$28871.1$}      \\
			
			\bottomrule[1pt]    
		\end{tabular}
	\end{table}

	\begin{figure}
		\centering
		\setkeys{Gin}{width=.7\linewidth}
		\includegraphics{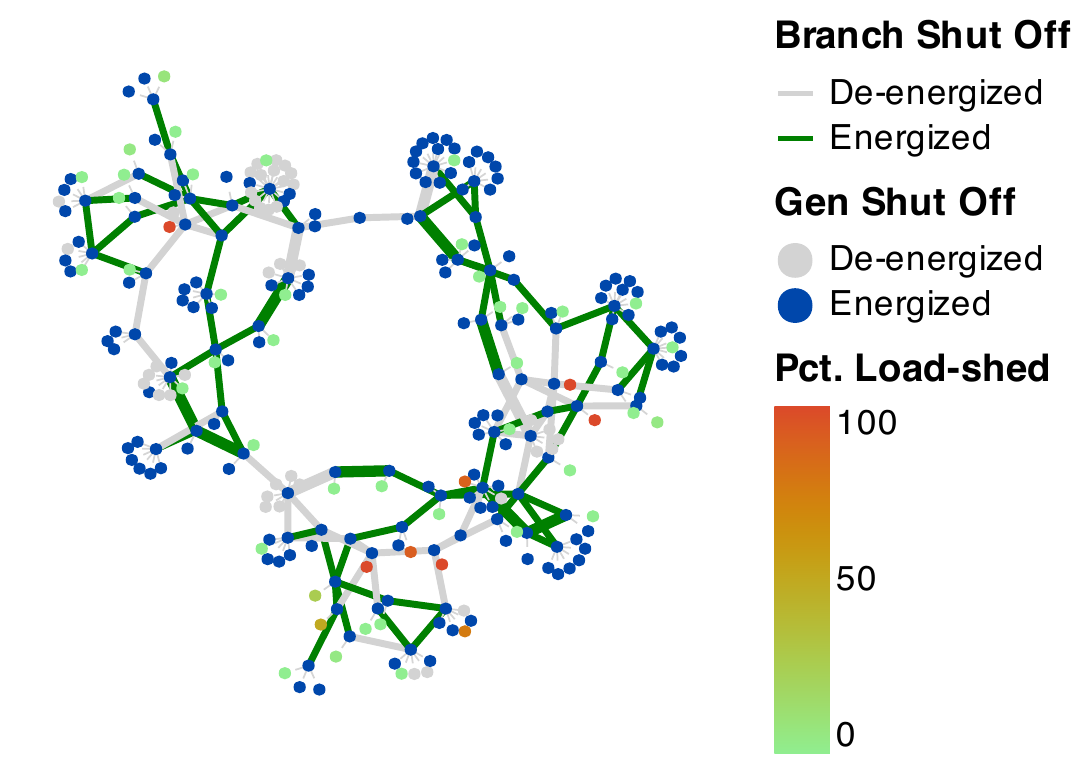}
		\caption{Illustration of the power system under the nominal plan $X^{ro}$ at time $t = 16$.}
		\label{fig:robustbenchmark}
	\end{figure}

	\subsection{Sensitivity Analysis of Disruption Probability} \label{subsec:sensitivity}
	We simulate the set of wildfire scenarios with a probabilistic model described in Section~\ref{subsec:scenario}, which then serves as the uncertainty model in the SMIP. It is necessary to examine the robustness of the optimal solution to the SMIP against the potential inaccuracy in such an uncertainty model. Therefore, we perform a sensitivity analysis in this section to check the out-of-sample performance of SMIP solutions obtained with different uncertainty assumptions.
	
	In Section~\ref{sec:model}, we construct a scenario set $\Omega^*$ in the SMIP model~\eqref{prob:stage1}, each scenario with an equal probability of occurrence, $1/|\Omega^*|$. Such a scenario set can be partitioned as $\Omega^* = \Omega^*_D\cup\Omega^*_N$ , where $\Omega^*_D$ $(\Omega^*_N)$ represents the set of scenarios with(out) a disruption. When there is no disruption, the nominal plan is carried out over the entire time horizon. Therefore, the partition is equivalent to $\Omega^* = \Omega^*_D \cup \{{\omega_0}\}$ with probabilities $p^{\omega_0} = |\Omega^*_N|/|\Omega^*|$ and $p^\omega = (1-p^{\omega_0})/|\Omega^*_D|$ for $\omega \in \Omega^*_D$, where ${\omega_0}$ is a scenario without disruption. 
	
	For the sensitivity analysis, we add a perturbation $\Delta p$ to $p^{\omega_0}$ and adjust \(p^\omega\) for \(\omega \in \Omega^*_D\) accordingly. We solve the SMIP for each \(\Delta p\) and test its solution's out-of-sample performance over the scenario set \(\tilde{\Omega}\), displayed in Table~\ref{table: cost for different delta p}. The out-of-sample no-disruption probability \(|\tilde{\Omega}_N|/|\tilde{\Omega}|\) is approximately \(0.05\), similar to that of \(\Omega^*\).
	
	We observe from Table~\ref{table: cost for different delta p} that the optimal SMIP solution can yield a decent out-of-sample performance even when the input perturbation is quadruple the original no-disruption probability (\(\Delta p = 0.2\)). This indicates that the SMIP optimal solution is robust against input inaccuracy. When \(\Delta p\) keeps increasing, the SMIP values a nominal plan that shuts off fewer components and thus minimizes the {load-shedding cost} under scenarios without wildfire disruption.  It leads to higher out-of-sample load-shedding and damage costs simultaneously as insufficient shut-offs are executed to prevent endogenous wildfires. The sensitivity analysis results show that the SMIP optimal solution is reliable as long as the input uncertainty model is not too liberal when estimating the wildfire risk.
	
	\begin{table}[ht]
		\centering
		\caption{Load-shedding costs and damage costs of nominal plans $X^{\Delta p}$ under non-disruptive and disruptive scenarios. }
		\label{table: cost for different delta p}
		\begin{tabular}{c|c|cc|cc}\toprule[1pt]\hline
			& \multicolumn{1}{c|}{Non-Disr.}  & \multicolumn{2}{c|}{Disruptive} &    \multicolumn{2}{c}{{Total Cost}}          \\ 
			$\Delta p$& Load shed  & Load shed & \multicolumn{1}{c|}{Damage}       &    $g_n(\cdot)$    &   RRI  \\ \cmidrule{1-6} 
			$-0.05$   & $1805.1$                   & $2553.0$              & $1242.7$      & $3782.8$ & $4.5\%$     \\
			$-0.015$  & $1741.5$                   & $2513.7$              & $1353.6$      & $3787.3$ & $4.6\%$      \\
			$0.0/X^*$ & $1678.2$                   & $2355.2$              & $1333.6$      & $3612.7$ & $0.0\%$      \\
			$0.015$   & $1660.0$                   & $2343.9$              & $1341.7$      & $3609.7$ & $-0.1\%$      \\
			$0.05$    & $1504.9$                   & $2360.2$              & $1325.4$      & $3609.6$ & $-0.1\%$      \\ 
			$0.1$     & $1250.7$                   & $2344.8$              & $1317.1$      & $3586.5$ & $-0.7\%$      \\ 
			$0.2$     & $828.4$                    & $2196.5$              & $1578.8$      & $3697.5$ & $2.3\%$      \\  
			$0.3$     & $207.3$                    & $2181.1$              & $1752.3$      & $3882.5$ & $7.3\%$      \\  
			$0.5$     & $185.1$                    & $2343.5$              & $1845.6$      & $4109.1$ & $13.5\%$      \\
			$0.7$     & $109.7$                    & $2410.8$              & $1851.0$      & $4153.3$ & $15.0\%$      \\
			$0.9$     & $99.0$                     & $2476.8$              & $1878.1$      & $4276.7$ & $18.4\%$      \\\hline  
			\bottomrule[1pt]    
		\end{tabular}
	\end{table}
	
	\subsection{Interaction between Endogenous and Exogenous Wildfires} \label{subsec:interaction}
	In this section, we investigate how endogenous and exogenous wildfires interact and affect the shut-off plans. Intuitively speaking, considering endogenous wildfires tends to encourage more shut-offs to prevent faults ignited by power system components, while the presence of exogenous wildfires introduces additional load-shedding costs which can be mitigated by keeping more components energized. To quantify this interaction, we compare three nominal plans: 1) $X^{exo}$ obtained by solving model~\eqref{prob:stage1} with $500$ scenarios that only contain exogenous wildfires; 2) $X^{end}$ obtained by solving model~\eqref{prob:stage1} with $500$ scenarios that only contain endogenous wildfires; and 3) \(X^*\) obtained in Section~\ref{subsec:benchmark} with $500$ scenarios that contain a mixture of exogenous and endogenous wildfires. Fig.~\ref{fig:interaction} illustrates the power system status, highlighting the de-energized components and load-shed in period \(t = 16\). The nominal plan \(X^{exo}\), shown in the left plot, exhibits the fewest number of de-energized components. In contrast, the right plot shows the nominal plan \(X^{end}\) leads to the highest number of shut-offs. The nominal plan \(X^*\), considering both exogenous and endogenous wildfires, represents a trade-off between the two. We perform three additional out-of-sample tests to evaluate those nominal plans with $5,000$ scenarios of: 1) exogenous wildfires only; 2) endogenous wildfires only; and 3) a mixture of endogenous and exogenous wildfires, as shown in \tcb{Fig.}~\ref{fig:interaction}. The nominal plan $X^{exo}$ exhibits inferior performance in both load-shedding and damage costs when endogenous wildfires are possible. The nominal plan \(X^{end}\) becomes conservative and shuts off too many components, resulting in consistently high load-shedding costs. While the nominal plan $X^*$ achieves a balance in between. Our finding confirms the intuition of how exogenous and endogenous wildfires interact and highlights the importance of modeling both types of wildfires in power system operations.
	
	\begin{figure*}
		\centering
		\setkeys{Gin}{width=\linewidth}
		\begin{subfigure}{0.32\textwidth}
			\includegraphics{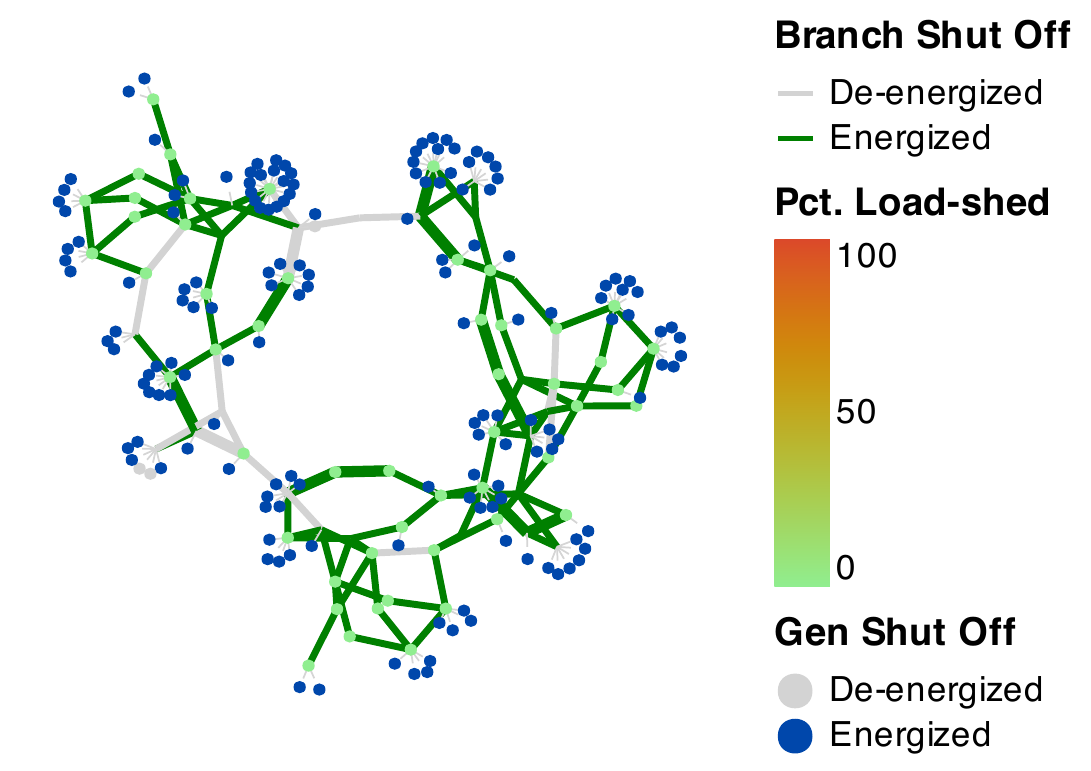}
			\caption{$15$ transmission lines are de-energized in $X^{exo}$ in period $16$.}
		\end{subfigure}
		\hfil
		\begin{subfigure}{0.32\textwidth}
			\includegraphics{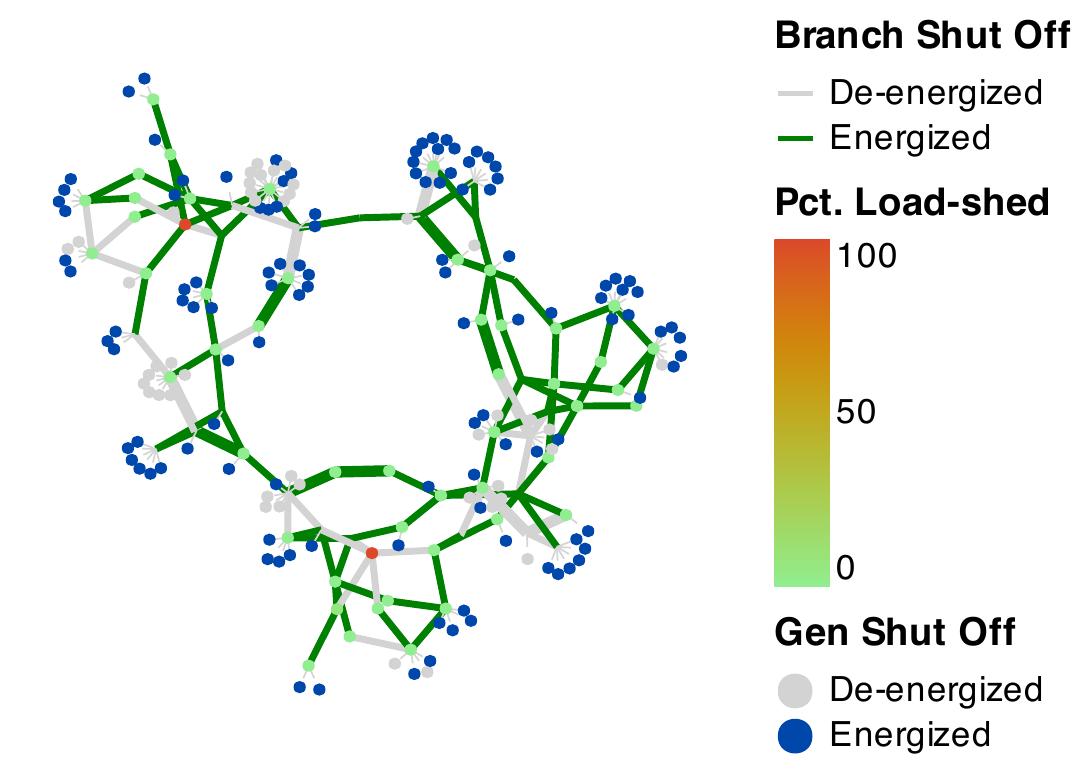}
			\caption{$26$ transmission lines are de-energized in $X^{*}$ in period $16$.}
		\end{subfigure}
		\hfil
		\begin{subfigure}{0.32\textwidth}
			\includegraphics{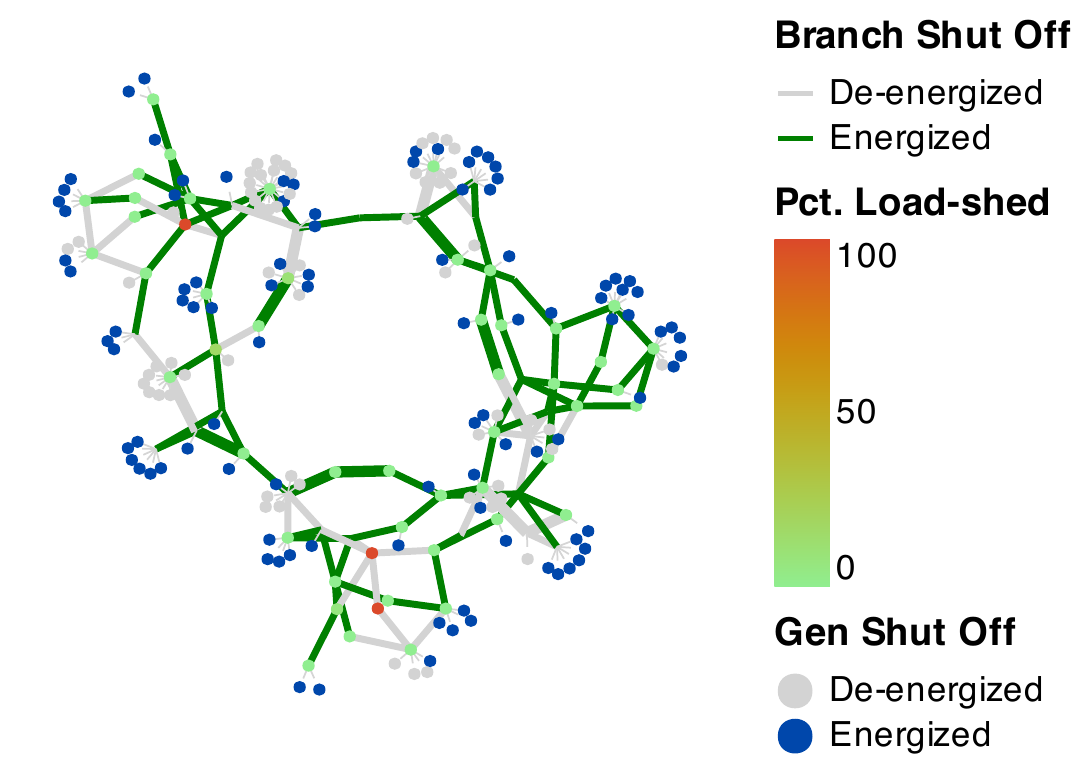}
			\caption{$31$ transmission lines are de-energized in $X^{end}$ in period $16$.}
		\end{subfigure}
		\caption{Visualization of interaction between endogenous and exogenous wildfire risk.}
		\label{fig:interaction}
	\end{figure*}
	
	\begin{table*}[ht]
		\centering
		\caption{{The performance of the nominal plans generated from different scenarios under different test sets. } }
		\label{table: interaction}
		\begin{tabular}{c|cccc||cccc||cccc}\toprule[1pt]\hline
			&  \multicolumn{4}{c||}{{Exogenous Wildfire Scenario Set}} & \multicolumn{4}{c||}{{Mixture Wildfire Scenario Set}} & \multicolumn{4}{c}{{Endogenous Wildfire Scenario Set}} \\\cmidrule{1-13} 
			& \multicolumn{1}{c}{{Non-Disr.}}  & \multicolumn{2}{c}{{Disruptive}} &  {Total} & \multicolumn{1}{c}{{Non-Disr.}}  & \multicolumn{2}{c}{{Disruptive}} & {Total} & \multicolumn{1}{c}{{Non-Disr.}}  & \multicolumn{2}{c}{{Disruptive}} &  {Total} \\ 
			{Type} & {Load shed}  & {Load shed} & \multicolumn{1}{c}{{Damage}}  & {$g_n(\cdot)$}  & {Load shed}  & {Load shed} & \multicolumn{1}{c}{{Damage}}  & {$g_n(\cdot)$} & {Load shed}  & {Load shed} & \multicolumn{1}{c}{{Damage}}  & {$g_n(\cdot)$} \\ \cmidrule{1-13} 
			
			{$X^{exo}$}   & {$0.0$}         & {$707.1$}     & {$636.5$}   & {$1289.8$}   & {$0.0$}        & {$3231.8$}        & {$2961.0$}      & {$5960.2$}      & {$0.0$}      & {$3380.4$}     & {$2979.8$}    & {$6106.0$}   \\
			
			{$X^{*\quad}$}   & {$1678.1$}      & {$1892.8$}     & {$636.5$}   & {$2495.2$}   & {$1678.2$}     & {$2355.2$}      & {$1333.6$}        & {$3612.7$}      & {$1678.2$}   & {$3107.1$}    & {$1814.0$}    & {$4792.8$}   \\
			
			{$X^{end}$}   & {$1998.5$}      & {$2249.9$}     & {$636.5$}   & {$2850.9$}   & {$1998.5$}     & {$2748.6$}        & {$1240.3$}      & {$3909.1$}      & {$1998.5$}    & {$3147.0$}   & {$1514.0$}    & {$4554.4$}   \\ \hline  
			\bottomrule[1pt]    
		\end{tabular}
	\end{table*}

\section{Conclusions \& Future Work} \label{sec:conclusion}
	This work contributes to the literature on stochastic optimization for power system operation under wildfire risk. We propose a novel two-stage stochastic program that captures the uncertainty of wildfire disruptions and temporal dynamics. We also develop an efficient decomposition algorithm that exploits the binary structure of state variables and generates valid cuts for solving large-scale instances. Our numerical experiments demonstrate the effectiveness and robustness of our approach in terms of solution quality and computational speed. They illustrate the benefits of the stochastic programming model by comparing it with deterministic counterparts \tcb{and the scenario-based robust optimization model.}
	
	For future work, we plan to extend our research in several directions. One direction is to investigate a theory-driven approach for generating effective and efficient cuts in our decomposition algorithm. Another direction is to expand the current modeling setting: 1) to relax the assumption of at most one disruption and extend our model to a multistage problem; 2) to include the unit commitment decisions, which would require an algorithmic development to handle the continuous state variables to model the ramping constraints; 3) to formulate a realistic multi-period uncertainty set/ambiguity set \tcb{based on high-fidelity wildfire simulators} and model the wildfire disruption in robust/distributionally robust optimization. A third direction is to address the challenge of ensuring feasibility with AC power flow when using a linear approximation of power flow in our model. Finally, we aim to develop a rolling horizon framework that can handle the wildfire forecast updates and dynamically re-optimize the de-energization schedule.

\bibliographystyle{IEEEtran}
\bibliography{ms_wildfire}

\begin{thebibliography}{10}
\providecommand{\url}[1]{#1}
\csname url@samestyle\endcsname
\providecommand{\newblock}{\relax}
\providecommand{\bibinfo}[2]{#2}
\providecommand{\BIBentrySTDinterwordspacing}{\spaceskip=0pt\relax}
\providecommand{\BIBentryALTinterwordstretchfactor}{4}
\providecommand{\BIBentryALTinterwordspacing}{\spaceskip=\fontdimen2\font plus
\BIBentryALTinterwordstretchfactor\fontdimen3\font minus
  \fontdimen4\font\relax}
\providecommand{\BIBforeignlanguage}[2]{{%
\expandafter\ifx\csname l@#1\endcsname\relax
\typeout{** WARNING: IEEEtran.bst: No hyphenation pattern has been}%
\typeout{** loaded for the language `#1'. Using the pattern for}%
\typeout{** the default language instead.}%
\else
\language=\csname l@#1\endcsname
\fi
#2}}
\providecommand{\BIBdecl}{\relax}
\BIBdecl

\bibitem{Muhs2020}
J.~W. Muhs, M.~Parvania, and M.~Shahidehpour, ``Wildfire risk mitigation: A
  paradigm shift in power systems planning and operation,'' \emph{IEEE Open
  Access Journal of Power and Energy}, vol.~7, pp. 366--375, 2020.

\bibitem{Holmes2008}
T.~P. Holmes, R.~J. Huggett~Jr, and A.~L. Westerling, ``Statistical analysis of
  large wildfires,'' \emph{The Economics of Forest Disturbances}, vol.~79, pp.
  59--77, 2008.

\bibitem{vazquez2022wildfire}
D.~A.~Z. Vazquez, F.~Qiu, N.~Fan, and K.~Sharp, ``Wildfire mitigation plans in
  power systems: A literature review,'' \emph{IEEE Transactions on Power
  Systems}, vol.~37, no.~5, pp. 3540--3551, 2022.

\bibitem{2019wildfirerisk}
\BIBentryALTinterwordspacing
``2019 wildfire risk report,'' CoreLogic, Tech. Rep., 2019. [Online].
  Available:
  \url{https://www.corelogic.com/wp-content/uploads/sites/4/downloadable-docs/wildfire-report_0919-01-screen.pdf}
\BIBentrySTDinterwordspacing

\bibitem{choobineh2015vulnerability}
M.~Choobineh, B.~Ansari, and S.~Mohagheghi, ``Vulnerability assessment of the
  power grid against progressing wildfires,'' \emph{Fire Safety Journal},
  vol.~73, pp. 20--28, 2015.

\bibitem{WildfireManagement}
S.~Jazebi, F.~de~Le\'{o}n, and A.~Nelson, ``Review of wildfire management
  techniques—{Part I}: Causes, prevention, detection, suppression, and data
  analytics,'' \emph{IEEE Transactions on Power Delivery}, vol.~35, no.~1, pp.
  430--439, 2020.

\bibitem{PSPS}
\BIBentryALTinterwordspacing
``Utility public safety power shutoff plans (de-energization).'' California
  Public Utilities Commission (CPUC), 2022. [Online]. Available:
  \url{https://www.cpuc.ca.gov/psps/}
\BIBentrySTDinterwordspacing

\bibitem{rhodes2020balancing}
N.~Rhodes, L.~Ntaimo, and L.~Roald, ``Balancing wildfire risk and power outages
  through optimized power shut-offs,'' \emph{IEEE Transactions on Power
  Systems}, vol.~36, no.~4, pp. 3118--3128, 2021.

\bibitem{taylor2022framework}
S.~Taylor and L.~A. Roald, ``A framework for risk assessment and optimal line
  upgrade selection to mitigate wildfire risk,'' \emph{Electric Power Systems
  Research}, vol. 213, p. 108592, 2022.

\bibitem{muhs2020wildfire}
J.~W. Muhs, M.~Parvania, and M.~Shahidehpour, ``Wildfire risk mitigation: A
  paradigm shift in power systems planning and operation,'' \emph{IEEE Open
  Access Journal of Power and Energy}, vol.~7, pp. 366--375, 2020.

\bibitem{davoudi2021reclosing}
M.~Davoudi, B.~Efaw, M.~Avenda{\~n}o-Mora, J.~L. Lauletta, and G.~B. Huffman,
  ``Reclosing of distribution systems for wildfire prevention,'' \emph{IEEE
  Transactions on Power Delivery}, vol.~36, no.~4, pp. 2298--2307, 2021.

\bibitem{DataMiningPreventionWildfire}
T.~Zhou, B.~Li, C.~Wu, Y.~Tan, L.~Mao, and W.~Wu, ``Studies on big data mining
  techniques in wildfire prevention for power system,'' in \emph{2019 IEEE 3rd
  Conference on Energy Internet and Energy System Integration (EI2)}, 2019, pp.
  866--871.

\bibitem{bayani2022quantifying}
R.~Bayani, M.~Waseem, S.~D. Manshadi, and H.~Davani, ``Quantifying the risk of
  wildfire ignition by power lines under extreme weather conditions,''
  \emph{IEEE Systems Journal}, vol.~17, no.~1, pp. 1024--1034, 2022.

\bibitem{Hong2022data}
W.~Hong, B.~Wang, M.~Yao, D.~Callaway, L.~Dale, and C.~Huang, ``Data-driven
  power system optimal decision making strategy under wildfire events,''
  Lawrence Livermore National Lab, Tech. Rep. LLNL-CONF-831390, 2022.

\bibitem{huang2023review}
C.~Huang, Q.~Hu, L.~Sang, D.~D. Lucas, R.~Wong, B.~Wang, W.~Hong, M.~Yao, and
  V.~Donde, ``A review of public safety power shutoffs (psps) for wildfire
  mitigation: Policies, practices, models and data sources,'' \emph{IEEE
  Transactions on Energy Markets, Policy and Regulation}, 2023.

\bibitem{DP_PSPS}
A.~Lesage-Landry, F.~Pellerin, J.~A. Taylor, and D.~S. Callaway, ``Optimally
  scheduling public safety power shutoffs,'' \emph{arXiv preprint
  arXiv:2203.02861}, 2022.

\bibitem{bayani2023resilient}
R.~Bayani and S.~D. Manshadi, ``Resilient expansion planning of electricity
  grid under prolonged wildfire risk,'' \emph{IEEE Transactions on Smart Grid},
  2023.

\bibitem{zhou2023optimal}
Y.~Zhou, K.~Sundar, D.~Deka, and H.~Zhu, ``Optimal power system topology
  control under uncertain wildfire risk,'' \emph{arXiv preprint
  arXiv:2303.07558}, 2023.

\bibitem{Noah2022}
N.~Rhodes and L.~Roald, ``Co-optimization of power line de-energization and
  restoration under high wildfire ignition risk,'' \emph{arXiv preprint
  arXiv:2204.02507}, 2022.

\bibitem{kody2022sharing}
A.~Kody, A.~West, and D.~K. Molzahn, ``Sharing the load: Considering fairness
  in de-energization scheduling to mitigate wildfire ignition risk using
  rolling optimization,'' in \emph{2022 IEEE 61st Conference on Decision and
  Control (CDC)}.\hskip 1em plus 0.5em minus 0.4em\relax IEEE, 2022, pp.
  5705--5712.

\bibitem{ke2022managing}
G.~Y. Ke, ``Managing reliable emergency logistics for hazardous materials: A
  two-stage robust optimization approach,'' \emph{Computers \& Operations
  Research}, vol. 138, p. 105557, 2022.

\bibitem{WildfireManagement2}
S.~Jazebi, F.~de~Le\'{o}n, and A.~Nelson, ``Review of wildfire management
  techniques—{Part II}: Urgent call for investment in research and
  development of preventative solutions,'' \emph{IEEE Transactions on Power
  Delivery}, vol.~35, no.~1, pp. 440--450, 2020.

\bibitem{Li2021SpatialAT}
S.~Li and T.~Banerjee, ``Spatial and temporal pattern of wildfires in
  {California} from 2000 to 2019,'' \emph{Scientific Reports}, vol.~11, no.
  1:8779, 2021.

\bibitem{nagpal2018wildfire}
M.~Nagpal, R.~P. Barone, T.~G. Martinich, Z.~Jiao, S.-H. Manuel, and
  S.~Merriman, ``Wildfire trips de-energized line shunt reactor,'' \emph{IEEE
  Transactions on Power Delivery}, vol.~34, no.~2, pp. 760--768, 2018.

\bibitem{pico2019transient}
H.~N.~V. Pico and B.~B. Johnson, ``Transient stability assessment of
  multi-machine multi-converter power systems,'' \emph{IEEE Transactions on
  Power Systems}, vol.~34, no.~5, pp. 3504--3514, 2019.

\bibitem{mitchell2013power}
J.~W. Mitchell, ``Power line failures and catastrophic wildfires under extreme
  weather conditions,'' \emph{Engineering Failure Analysis}, vol.~35, pp.
  726--735, 2013.

\bibitem{Scott2013AWR}
J.~H. Scott, M.~P. Thompson, and D.~E. Calkin, ``A wildfire risk assessment
  framework for land and resource management,'' U.S. Department of Agriculture,
  Forest Service, Rocky Mountain Research Station, Tech. Rep. RMRS-GTR-315,
  2013.

\bibitem{8620983}
S.~Dian, P.~Cheng, Q.~Ye, J.~Wu, R.~Luo, C.~Wang, D.~Hui, N.~Zhou, D.~Zou,
  Q.~Yu, and X.~Gong, ``Integrating wildfires propagation prediction into early
  warning of electrical transmission line outages,'' \emph{IEEE Access},
  vol.~7, pp. 27\,586--27\,603, 2019.

\bibitem{BirgLouv97}
J.~R. Birge and F.~Louveaux, \emph{Introduction to Stochastic
  Programming}.\hskip 1em plus 0.5em minus 0.4em\relax New York, NY, USA:
  Springer-Verlag, 1997.

\bibitem{HaoxiangIJOC}
H.~Yang and H.~Nagarajan, ``Optimal power flow in distribution networks under
  {N-1} disruptions: A multistage stochastic programming approach,''
  \emph{INFORMS Journal on Computing}, vol.~34, no.~2, pp. 690--709, 2022.

\bibitem{zou2016nested}
J.~Zou, S.~Ahmed, and X.~A. Sun, ``Stochastic dual dynamic integer
  programming,'' \emph{Mathematical Programming}, vol. 175, no. 1-2, pp.
  461--502, 2019.

\bibitem{CellularAutomata}
J.~Freire and C.~Dacamara, ``Using cellular automata to simulate wildfire
  propagation and to assist in fire prevention and fighting,'' \emph{Natural
  Hazards and Earth System Sciences Discussions}, pp. 1--17, 2018.

\bibitem{RTS-GLMC}
C.~Barrows, A.~Bloom, A.~Ehlen, J.~Ikäheimo, J.~Jorgenson, D.~Krishnamurthy,
  J.~Lau, B.~McBennett, M.~O’Connell, E.~Preston, A.~Staid, G.~Stephen, and
  J.-P. Watson, ``The {IEEE} reliability test system: A proposed 2019 update,''
  \emph{IEEE Transactions on Power Systems}, vol.~35, no.~1, pp. 119--127,
  2020.

\bibitem{maptools}
\BIBentryALTinterwordspacing
MapTools, ``Why use {UTM} coordinates,'' 2023. [Online]. Available:
  \url{https://www.maptools.com/tutorials/utm/why_use_utm}
\BIBentrySTDinterwordspacing

\bibitem{WFPI}
\BIBentryALTinterwordspacing
``Wildland fire potential index {(WFPI)}.'' [Online]. Available:
  \url{https://www.usgs.gov/fire-danger-forecast/wildland-fire-potential-index-wfpi}
\BIBentrySTDinterwordspacing

\bibitem{Alexandridis2008}
A.~Alexandridis, D.~Vakalis, C.~Siettos, and G.~Bafas, ``A cellular automata
  model for forest fire spread prediction: The case of the wildfire that swept
  through {Spetses} island in 1990,'' \emph{Applied Mathematics and
  Computation}, vol. 204, no.~1, pp. 191--201, 2008.

\bibitem{PoissionRegression}
S.~Yang, W.~Zhou, S.~Zhu, L.~Wang, L.~Ye, X.~Xia, and H.~Li, ``Failure
  probability estimation of overhead transmission lines considering the spatial
  and temporal variation in severe weather,'' \emph{Journal of Modern Power
  Systems and Clean Energy}, vol.~7, no.~1, pp. 131--138, 2019.

\bibitem{lemarechal1995new}
C.~Lemar{\'e}chal, A.~Nemirovskii, and Y.~Nesterov, ``New variants of bundle
  methods,'' \emph{Mathematical programming}, vol.~69, no.~1, pp. 111--147,
  1995.

\bibitem{TWRAS}
\BIBentryALTinterwordspacing
``Texas {A\&M} forest service: {Texas} wildfire risk assessment portal.''
  [Online]. Available: \url{https://texaswildfirerisk.com/}
\BIBentrySTDinterwordspacing

\bibitem{NuclearPower}
\BIBentryALTinterwordspacing
{World Nuclear Association}, ``Economics of nuclear power,'' 2022. [Online].
  Available:
  \url{https://world-nuclear.org/information-library/economic-aspects/economics-of-nuclear-power.aspx}
\BIBentrySTDinterwordspacing

\bibitem{WindPower}
\BIBentryALTinterwordspacing
D.~Blewett, ``Wind turbine cost: {H}ow {M}uch? {A}re {T}hey {W}orth {I}t {I}n
  2022?'' 2022. [Online]. Available:
  \url{https://weatherguardwind.com/how-much-does-wind-turbine-cost-worth-it/}
\BIBentrySTDinterwordspacing

\bibitem{NuclearFinancing}
\BIBentryALTinterwordspacing
{World Nuclear Association}, ``Financing nuclear energy,'' 2020. [Online].
  Available:
  \url{https://world-nuclear.org/information-library/economic-aspects/financing-nuclear-energy.aspx}
\BIBentrySTDinterwordspacing

\bibitem{TransmissionCost}
\BIBentryALTinterwordspacing
``2021 transmission cost report.'' Australian Energy Market Operator, Tech.
  Rep., 2021. [Online]. Available:
  \url{https://aemo.com.au/-/media/files/major-publications/isp/2021/transmission-cost-report.pdf?la=en}
\BIBentrySTDinterwordspacing

\bibitem{JuMP}
I.~Dunning, J.~Huchette, and M.~Lubin, ``{JuMP}: A modeling language for
  mathematical optimization,'' \emph{SIAM Review}, vol.~59, no.~2, pp.
  295--320, 2017.

\bibitem{Julia}
S.~K. Jeff~Bezanson, Alan~Edelman and V.~Shah, ``Julia: A fresh approach to
  numerical computing,'' \emph{SIAM Review}, 2014.

\bibitem{gurobi}
\BIBentryALTinterwordspacing
{Gurobi Optimization, Inc.}, \emph{Gurobi optimizer reference manual.}, 2014.
  [Online]. Available: \url{https://www.gurobi.com}
\BIBentrySTDinterwordspacing

\bibitem{PowerModels}
C.~Coffrin, R.~Bent, K.~Sundar, Y.~Ng, and M.~Lubin, ``Powermodels.jl: An
  open-source framework for exploring power flow formulations,'' in \emph{2018
  Power Systems Computation Conference (PSCC)}, June 2018, pp. 1--8.

\bibitem{Agents}
G.~Datseris, A.~R. Vahdati, and T.~C. DuBois, ``Agents.jl: a performant and
  feature-full agent-based modeling software of minimal code complexity,''
  \emph{Simulation}, p. 003754972110688, 2022.

\bibitem{yang2021robust}
H.~Yang, D.~P. Morton, C.~Bandi, and K.~Dvijotham, ``Robust optimization for
  electricity generation,'' \emph{INFORMS Journal on Computing}, vol.~33,
  no.~1, pp. 336--351, 2021.

\end{thebibliography}
\end{document}